\documentclass[11pt,a4paper]{article}

\usepackage{import}
\usepackage{geometry}
\usepackage{mystyle}
\usepackage{style_statistics}
\crefname{question}{question}{questions}

\usepackage[caption=false, position = bottom]{subfig}
\crefformat{equation}{(#2#1#3)}
\crefrangeformat{equation}{(#3#1#4)--(#5#2#6)}
\crefrangeformat{enumi}{#3#1#4--#5#2#6}
\crefmultiformat{equation}{(#2#1#3)}%
{ and~(#2#1#3)}{, (#2#1#3)}{ and~(#2#1#3)}

\DeclarePairedDelimiter\ceil{\lceil}{\rceil}

\crefformat{equation}{(#2#1#3)}
\crefrangeformat{equation}{(#3#1#4)--(#5#2#6)}

\usepackage{tcolorbox}

\usepackage{doi}
\usepackage{graphicx}
\usepackage{csquotes}

\date{\today}

\title{How many samples are needed to reliably approximate the best linear estimator for a linear inverse problem? }

\newcommand{\keywords}[1]{\textbf{Key words.}\quad #1}
\newcommand{\msc}[1]{2020 Mathematics Subject Classification.\quad #1}

\begin{document}
\author{Gernot Holler\footnotemark[1]}

\footnotetext[1]{Institute of Mathematics and Scientific Computing, University of Graz,
  Heinrichstrasse 36, 8010 Graz, Austria, email: \texttt{gernot.holler@gmail.com} .}

\maketitle

\begin{abstract}
The linear minimum mean squared error (LMMSE) estimator is the best linear estimator for a Bayesian linear inverse problem with respect to the mean squared error.
It arises as the solution operator to a Tikhonov-type regularized inverse problem with a particular quadratic discrepancy term and a particular quadratic regularization operator.
To be able to evaluate the LMMSE estimator, one must know the forward operator and the first two statistical moments of both the prior and the noise. If such knowledge is not available, one may approximate the LMMSE estimator based on given samples. In this work, it is investigated, in a finite-dimensional setting,  how many samples are needed to reliably approximate the LMMSE estimator, in the sense that, with high probability, the mean squared error of the approximation is smaller than a given multiple of the mean squared error of the LMMSE estimator.  
\end{abstract}
\msc{62J05L, 62H12, 62F15} \\ 

\section{Introduction}
The objective in a finite-dimensional linear inverse problem with additive noise is to recover an unknown parameter $x \in \mathbb{R}^N$ from data $y \in \mathbb{R}^M$ of the form
\begin{equation}\label{eq:intro_form_of_data}
y = Ax + z.
\end{equation}
The matrix $A \in \mathbb{R}^{M \times N}$ represents the forward operator $x \mapsto Ax$, and $z \in \mathbb{R}^M$ is the noise. Estimators for inverse problems are functions $f \colon \mathbb{R}^M \to \mathbb{R}^N $ that map the data to an estimate of the parameter such that $f(y) \approx x$. Usually, estimators are designed based on assumptions about the parameter and the noise, and defined implicitly as solution operators to minimization problems that depend on the data and the forward operator \cite{Engl1996, tarantola2005inverse}. We consider the statistical model of an inverse problem \cite{kaipio2006statistical}, where $x$ and $z$ are modeled as realizations of independent random variables $X$ and $Z$; the data $y$ is then a realization of the random variable $Y \coloneqq AX + Z$. We additionally assume that the means of $X$ and $Z$ are zero.
We investigate how many independent samples of parameter-data pairs $(X,Y)$ are needed to learn a linear estimator $f$ with a small mean squared error  
\begin{equation} 
\text{MSE}(f) \coloneqq \ev{ \| f(Y) - X \|_2^2}.
\end{equation}
If the covariance matrix $C_{XX}$ of $X$, the covariance matrix $C_{ZZ}$ of $Z$, and the forward operator $A$ are given, we do not need learning to find a linear estimator with a small mean squared error.  This is because the linear estimator with the smallest mean squared error---the so-called linear minimum mean squared error (LMMSE) estimator---is given by
\begin{equation}\label{eq:LMMSE_estimator}
f_{\text{LMMSE}}(Y) = (A^\top C_{ZZ}^{-1} A + C_{XX}^{-1})^{-1} A^\top C_{ZZ}^{-1} Y;
\end{equation}
see \Cref{prop:characterization_of_LMMSE}. Its mean squared error is given by $\trace{C_{EE}}$ for
\begin{equation*}
C_{EE} \coloneqq C_{XX} - C_{XX} A^\top (A C_{XX} A^\top + C_{ZZ})^{-1} A C_{XX}. 
\end{equation*} 
If $C_{XX}$, $C_{ZZ}$ or $A$ are unknown, then the LMMSE estimator cannot be evaluated and it makes sense to learn linear estimators. We assume that both $C_{XX}$ and $C_{ZZ}$ as well as $A$ are unknown. Given $n$ independent samples $\{(X_i,Y_i)\}_{i=1}^n$ of $(X,Y)$, a simple method to learn a linear estimator is to solve the empirical mean squared error minimization problem
\begin{equation}\label{eq:EMSE}
\min \left\{ \frac{1}{n} \sum_{i=1}^n \|f(Y_i) - X_i\|_2^2\mid f \colon \mathbb{R}^M \to \mathbb{R}^N \text{ is linear} \right\}.
\end{equation}
We call the solution $\hat{f}$ to \eqref{eq:EMSE} the least squares estimator (provided it is unique). The aim of this work is to determine, for given $\varepsilon > 0$, how many samples are needed to ensure that the mean squared error of the least squares estimator is not greater than $(1+\varepsilon)$ times the mean squared error of the LMMSE estimator. Since the least squares estimator is itself random (since it depends on random samples), to avoid ambiguity, we must distinguish whether this error bound should hold only in the expected value or with a certain probability. This distinction leads us to the following two questions.

\begin{question}[Expected mean squared error]\label{quest:avg_err} Given $\varepsilon >0$, how many samples $n$ are needed to ensure that the expected mean squared error of the least squares estimator is not greater than $(1+\varepsilon)$ times the mean squared error of the LMMSE estimator, \ie  
\begin{equation*}
\ev{\| \hat{f}(Y) - X \|_2^2}  \leq \trace{C_{EE}} \left( 1+ \varepsilon \right)?
\end{equation*}
\end{question}
\begin{question}[Tail bounds for the mean squared error]\label{quest:conf_err} Given $\varepsilon >0$ and $\delta \in (0,1)$, how many samples $n$ are needed to ensure that with probability at least $1-\delta$, the mean squared error of the least squares estimator is not greater than $(1+\varepsilon)$ times the mean squared error of the LMMSE estimator, \ie 
\begin{equation}\label{eq:req_quest2}
\ev_{X,Y}{\| \hat{f}(Y) - X \|_2^2}  \leq \trace{C_{EE}} \left( 1+ \varepsilon \right)
\end{equation}
with probability at least $1- \delta$?
\end{question} 

\subsection{The Gaussian model} We first answer \Cref{quest:avg_err,quest:conf_err} for the Gaussian model, \ie under the assumption that $X$ and $Z$ are independent zero-mean Gaussian random vectors. For the Gaussian model, the answer to \Cref{quest:avg_err} is that
\begin{equation}\label{eq:formula_n_gaussian}
n_0(M, \varepsilon) = \lceil{M/\varepsilon\rceil} + M + 1
\end{equation}
samples are needed; see \Cref{rem:expected_error_Gaussian_model}. Here, $\ceil{M/\varepsilon}$ is the least integer greater than or equal to $M/\varepsilon$. The formula in \eqref{eq:formula_n_gaussian} is interesting for two reasons: first, it gives the exact number of samples needed, not just an upper bound for it; therefore, it provides us with a useful benchmark for all subsequent results. Second, it depends only on the dimension of the data $M$ and the tolerance $\varepsilon$ for the relative excess error; it is independent of the dimension of the unknown parameter and the forward operator.

The requirement in \Cref{quest:conf_err} is more restrictive than that in \Cref{quest:avg_err}, at least for small $\delta$, since it requires that the mean squared error of the least squares estimator be smaller than $(1+\varepsilon)$ times the mean squared error of the LMMSE estimator with probability $1-\delta$ and not just in the expected value. However, it can be shown that if the number of samples is chosen according to \eqref{eq:formula_n_gaussian}, then the mean squared error of the least squares estimator concentrates sharply around its expected value as the dimension of the data $M$ tends to infinity; see \Cref{prop:Gaussian_error_converges_in_p}. This asymptotic result implies that for large $M$ the number of samples needed in \Cref{quest:conf_err} is approximately the same as the number of samples needed in \Cref{quest:avg_err}, even for small $\delta$. 
A shortcoming of this result is that it does not quantify how large $M$ must be for this to be true. To overcome this shortcoming, we derive a non-asymptotic tail bound for the mean squared error of the least squares estimator; see \Cref{thm:out_of_sample_error}. For the Gaussian model, this tail bound yields the following bound on the number of samples needed in \Cref{quest:conf_err} (see \Cref{rem:answer_question_2_Gauss}):
\begin{equation}\label{eq:sample_bound}
n_1(M, \varepsilon, \delta, \nu) = \left(\sqrt{\frac{M + 2 \sqrt{M \nu \ln{(3/\delta)} }    + 2 \nu \ln{(3/\delta)}}{\varepsilon}} + \sqrt{M} + \sqrt{\ln{(3/\delta)}} \right)^2. 
\end{equation}
Here, $\nu$ is defined as the ratio of the largest eigenvalue to the sum of all eigenvalues of the matrix $C_{EE}$, \ie $\nu \coloneqq \| C_{EE} \|_2 /\trace{C_{EE}}$, where $\| C_{EE} \|_2$  denotes the spectral norm of $C_{EE}$. Since the number of eigenvalues of $C_{EE}$ is equal to the dimension of $X$, we expect $\nu$ to decrease as the dimension of $X$ increases. The bound in \eqref{eq:sample_bound} then suggests that the mean squared error of the least squares estimator is more concentrated for high-dimensional parameters than for low-dimensional parameters; this suggestion is consistent with our experiments; see \Cref{fig:empirical_tail_mse_Gauss}. In practice, we prefer a bound that is independent of the (unknown) value of $\nu$. Fortunately, since $\nu \leq 1$, a simple calculation shows that
\begin{equation}\label{eq:no_samples_needed}
n_2(M, \varepsilon, \delta) = \left(M + 2 \sqrt{M \ln{(3/\delta)}}    + 2  \ln{(3/\delta)}\right) \left(1/\varepsilon +2/\sqrt{\varepsilon} + 1 \right)
\end{equation}
is an upper bound for the number of samples in \eqref{eq:sample_bound}. The formula in \eqref{eq:no_samples_needed} scales only logarithmically in $1/\delta$. For fixed $\delta \in (0,1)$, we have $n_2(M,\varepsilon,\delta) / n_0(M, \varepsilon) \to 1$ as $M \to \infty$ and $\varepsilon \to 0$, where $n_0(M, \varepsilon)$ is defined in \eqref{eq:formula_n_gaussian}. Hence, for large $M$ and small $\varepsilon>0$, the numbers of samples needed in \Cref{quest:avg_err,quest:conf_err} are approximately the same.

\subsection{The general (sub-Gaussian) model}
Our goal is to obtain similar answers as for the Gaussian model for more general models. The main contribution of this paper is \Cref{thm:out_of_sample_error}, which provides a non-asymptotic tail bound for the mean squared error of the least squares estimator under sub-Gaussian conditions. Under appropriate additional assumptions, this tail bound yields an upper bound for \Cref{quest:conf_err} for the general model that is similar to the upper bound for the Gaussian model in \eqref{eq:sample_bound}; see \Cref{rem:answer_question_2}. Moreover, it leads to a bound for the asymptotic mean squared error as $M$ tends to $\infty$; see \Cref{rem:asymptotic_error}. 

\subsection{Related work}
A consistency analysis for learning the regularization parameter in a regularized inverse problem with a quadratic regularization parameter is provided in \cite{chada2021consistency}. Asymptotic results for the eigenvalues of random matrices have been used before to investigate the asymptotic behavior of the mean squared error for least squares and ridge regression; see \eg \cite{hastie2019surprises, dobriban2018high} and the references given there. Our main result,  \Cref{thm:out_of_sample_error}, is partly inspired by \cite[Theorem 1]{hsu2012random}, which provides a non-asymptotic tail bound for the mean squared error of the least squares estimator in random design linear regression. While \cite[Theorem 1]{hsu2012random} covers only the case of a scalar response, we allow for multi-dimensional responses. Moreover, we believe that the hypotheses in \Cref{thm:out_of_sample_error} are more natural than the comparable hypotheses in \cite[Condition 1-3]{hsu2012random}. This is because Condition 1 and 3 in \cite{hsu2012random} both involve the random variable $C_{YY}^{-1/2} Y$ (when formulated in our notation), whereas in \Cref{thm:out_of_sample_error} the hypotheses on $C_{YY}^{-1/2} Y$ are separated from those on the conditional mean squared error of the minimum mean squared error (MMSE) estimator and the conditional distance of the LMMSE to the MMSE estimator.
Tail bounds for the mean squared error of the least squares estimator are provided also in \cite[Theorem 1.2]{oliveira2016lower} (under finite moment conditions), but also in this result $C_{YY}^{-1/2} Y$ is involved in multiple hypotheses.

\subsection{Outline}
In \Cref{sec:preliminaries}, we recall the necessary background on minimal mean squared error estimators. In \Cref{sec:lstsq}, we introduce notation and state the least squares problem. In \Cref{sec:asymptotic-mse}, we derive estimates for the expected value of the mean squared error of the least squares estimator. In \Cref{sec:non-asymptotic-mse}, we derive a non-asymptotic tail bound for the mean squared error of the least squares estimator. In \Cref{sec:num-exp}, we provide numerical experiments.

\section{Preliminaries}\label{sec:preliminaries}
Throughout this work, we let the parameter $X$ and the noise $Z$ be independent random variables taking values in $\mathbb{R}^N$ and $\mathbb{R}^M$. We make the following assumptions: $X$ and $Z$ are square integrable, the means of $X$ and $Z$ are zero, and the covariance matrices $C_{XX} \in \mathbb{R}^{N \times N}$ and $C_{ZZ} \in \mathbb{R}^{M \times M}$ of $X$ and $Z$ are invertible. Moreover, we suppose that the push-forward measures of $X$ and $Z$ have a continuous Radon--Nikod\'ym derivative with respect to the Lebesgue measure. The data is given by $Y \coloneqq A X + Z$ for $A \in \mathbb{R}^{M \times N}$. We let $\| \cdot \|_2$  denote the Euclidean norm, and let $\trace{\cdot}$ denote the trace of a matrix.

\subsection{The minimum mean squared error estimator}\label{subsec:MMSE_estimator}
Although we focus on linear estimators, to interpret the hypotheses of \Cref{thm:out_of_sample_error}, we need to recall that the estimator with the smallest mean squared error among all measurable functions (without the linearity requirement)---the so-called minimum mean squared error (MMSE) estimator---is equal to the conditional expected value of $X$, \ie
\begin{equation}\label{eq:MSE_alternate_form}
f_\text{MMSE}(Y) = \ev{X \mid Y}.
\end{equation}
Even though this fact is well-known, see \eg \cite[Section 11.4]{kay1993fundamentals} and \cite[Example 2.2.6 on p.~58]{casella2002}, we believe it is worthwhile to sketch how it can be proven: by the law of iterated expectations \cite[Theorem 3.24 on p.~55]{wasserman2013all}, we have
\begin{equation}\label{eq:iterated_MSE}
\ev{ \| f(Y) - X \|_2^2} = \ev_Y{\ev{ \|f(Y) - X\|_2^2 \mid Y} }.
\end{equation}
By iterating the expectations as in \eqref{eq:iterated_MSE}, we transform the problem of minimizing the mean squared error among all measurable functions $f$ to the pointwise problem of minimizing, for each given $Y$,  the conditional mean squared error $\ev{ \|f(Y) - X\|_2^2 \mid Y}$ among all estimates $f(Y)$. A simple calculation shows that the conditional expected value satisfies
\begin{equation*}
\ev{\langle \ev{X \mid Y} - X, v \rangle_2  \mid Y} = 0  \quad \text{for all } v \in \mathbb{R}^N,
\end{equation*}
where $\langle \cdot, \cdot \rangle_2$ is the Euclidean inner product. The Pythagorean theorem then yields
\begin{equation}\label{eq:pythagoras}
\ev{\| f(Y) - X \|_2^2 \mid Y} = \ev{\| \ev{X \mid Y} - X \|_2^2 \mid Y} + \| \ev{X \mid Y} - f(Y) \|_2^2.
\end{equation}
Since the last summand in \eqref{eq:pythagoras} is always nonnegative, and zero only if $f(Y)=\ev{X \mid Y}$, it follows that $\ev{X \mid Y}$ minimizes the conditional mean squared error for each $Y$. Using \eqref{eq:iterated_MSE} and the monotonicity of the expectation operator, we deduce that the conditional expected value minimizes the mean squared error among all measurable functions, which is what we wanted to show. The identity in \eqref{eq:pythagoras} has another important consequence: it shows that the conditional mean squared error of any estimate $f(Y)$ is equal to the sum of the conditional mean squared error of the MMSE estimate and the conditional squared distance of $f(Y)$ to the MMSE estimate. Hence, up to a constant, the mean squared error of any estimator depends only on its mean squared distance to the MMSE estimator. Unfortunately, the formula for the MMSE estimator \eqref{eq:MSE_alternate_form} can be evaluated only if the conditional expected value is known---which is usually not the case. As a simpler alternative, we consider minimizing the mean squared error among all linear functions.

\subsection{The linear minimum mean squared error estimator}
A minimizer of the mean squared error among all linear functions is called a linear minimum mean squared error (LMMSE) estimator. In the following proposition, we prove that the LMMSE estimator is unique and recall some of its basic properties. 
\begin{prop}[Characterization and properties of the LMMSE estimator]\label{prop:characterization_of_LMMSE}
We have
\begin{enumerate}[label=\roman*)]
\item the LMMSE estimator is unique and given by
\begin{equation*}
f_{\text{LMMSE}}(Y) = (A^\top C_{ZZ}^{-1} A + C_{XX}^{-1})^{-1} A^\top C_{ZZ}^{-1} Y,
\end{equation*} \label{item:LMMSE_unique_and_charact}
\item the LMMSE estimate $f_{\text{LMMSE}}(y) $ is the unique solution to 
\begin{equation*}
\min_{x \in \R^N}  \| A x - y\|_{C_{ZZ}^{-1/2}}^2 + \| x \|^2_{C_{XX}^{-1/2}},
\end{equation*}
where
\begin{equation*}
\| u \|_{C_{ZZ}^{-1/2}} \coloneqq \sqrt{\langle C_{ZZ}^{-1} u,  u \rangle_2} \quad \text{and} \quad \| v \|_{C_{XX}^{-1/2}} \coloneqq \sqrt{\langle C_{XX}^{-1} v,  v \rangle_2},
\end{equation*}\label{item:characterization_LMMSE_reginvprob}
\item the mean squared error of the LMMSE estimator is given by $\trace{C_{EE}}$, where 
\begin{equation*}
C_{EE}  = C_{XX} - C_{XX} A^\top (A C_{XX} A^\top + C_{ZZ})^{-1} A C_{XX}.
\end{equation*}\label{item:error_LMMSE}
\end{enumerate}
\end{prop}
\begin{proof}
The claims in \ref{item:LMMSE_unique_and_charact} and \ref{item:error_LMMSE} follow from \cite[Theorem 12.1 on p.~391]{kay1993fundamentals} and the identity
\begin{equation*}
C_{XX} A^\top (A C_{XX} A^\top + C_{ZZ})^{-1} = (A^\top C_{ZZ}^{-1} A + C_{XX}^{-1})^{-1} A^\top C_{ZZ}^{-1}  ,
\end{equation*}
The claim in \ref{item:characterization_LMMSE_reginvprob} follows by comparing the characterization in  \ref{item:LMMSE_unique_and_charact} with the first-order necessary and sufficient optimality conditions for the minimization problem in \ref{item:characterization_LMMSE_reginvprob}. 
\end{proof}
\Cref{prop:characterization_of_LMMSE} reveals that the LMMSE estimator and its mean squared error depend only on the first two statistical moments of both the parameter $X$ and the noise $Z$. The mean squared error of any linear estimator $f$ is equal to the sum of the mean squared error of the LMMSE estimator and the approximation error, \ie 
\begin{equation}
\ev{\| f(Y) - X  \|_2^2 } = \ev{\| f_{\text{LMMSE}}(Y) - X \|_2^2 } + \ev{\| f_{\text{LMMSE}}(Y) - f(Y) \|_2^2 }.
\end{equation}
This error decomposition holds because the LMMSE estimator is the orthogonal projection of the random variable $X$ onto the subspace of all linear functions depending only on $Y$. Since the mean squared error of the LMMSE estimator is $\trace{C_{EE}}$, we deduce that
\begin{equation}\label{MSE_decompostion_linear}
\ev{\| f(Y) - X  \|_2^2 } = \trace{C_{EE}} + \ev{\| f(Y)- f_{\text{LMMSE}}(Y)  \|_2^2 }.
\end{equation}
The error decomposition in \eqref{MSE_decompostion_linear} shows that $\ev{\| f(Y) - X  \|_2^2 } \leq (1+\varepsilon) \trace{C_{EE}}$ if and only if
$
\ev{\| f(Y) - f_{\text{LMMSE}}(Y) \|_2^2 } \leq \varepsilon \, \trace{C_{EE}}.
$
This observation allows us to restrict our attention to the approximation error $\ev{\| f(Y) - f_{\text{LMMSE}}(Y) \|_2^2 }$ in all subsequent results. 

\subsection{The estimation error}
We define the estimation error (of the LMMSE estimator) by $E \coloneqq X - f_\text{LMMSE}(Y)$. 
\begin{prop}\label{prop:estimation_error}
The conditional estimation error satisfies
\begin{align}
\ev{E \mid Y } &= f_\text{MMSE}(Y) - f_\text{LMMSE}(Y), \label{eq:ev_estimation_error} \\
\cov{E \mid Y} &= \ev{ (X - f_{\text{MMSE}}(Y))(X - f_{\text{MMSE}}(Y))^\top \mid Y}. \label{eq:cov_estimation_error}
\end{align}
\end{prop}
\begin{proof}
The identity in \eqref{eq:ev_estimation_error} follows from the characterization of the MMSE estimator in \eqref{eq:MSE_alternate_form}. 
The identity in \eqref{eq:cov_estimation_error} is by the translation invariance of the covariance operator. 
\end{proof}
 
The identity in \eqref{eq:ev_estimation_error} shows that the assumption that the $i$th component of the conditional expectation of the estimation error is bounded by a constant is equivalent to the assumption that the $i$th component of the difference between the LMMSE and the MMSE estimates is bounded by the same constant. The identity in \eqref{eq:cov_estimation_error} shows that the assumption that the trace of a matrix $C \in \mathbb{R}^{N \times N}$ is an upper bound for the trace of the conditional covariance matrix $\cov{E \mid Y}$ is equivalent to the assumption that the trace of $C$ is an upper bound for the conditional mean squared error of the MMSE estimator.

\subsection{The Gaussian model}
The Gaussian model of a linear inverse problem assumes that $X$ and $Z$ are zero-mean Gaussian random vectors. Two remarkable properties of the Gaussian model that do not apply in general are: 1.) the LMMSE and the MMSE estimators coincide, and 2.) the conditional estimation error is independent of $Y$. We verify these properties by recalling that $X \mid Y$ is a Gaussian random vector whose mean depends linearly on $Y$, and whose covariance matrix is independent of $Y$.

\begin{prop}[Posterior of the Gaussian model]\label{prop:posterior_gaussian_model}
Assume that $X$ and $Z$ are independent zero-mean Gaussian random vectors with invertible covariance matrices $C_{XX}$ and $C_{ZZ}$. Then $X \mid Y$ is a Gaussian random vector and  
\begin{align}
\ev{X \mid Y } &= C_{XX} A^\top ( A C_{XX} A^\top + C_{ZZ})^{-1} Y,\label{eq:mean_Gaussian} \\
\cov{X \mid Y } &= C_{XX} - C_{XX} A^\top (A C_{XX} A^\top + C_{ZZ})^{-1} A C_{XX} \label{eq:cov_Gaussian}
\end{align}
In particular, the LMMSE and the MMSE estimators coincide. Moreover, the conditional estimation error $(E \mid Y)  = (X\mid Y) - f_{\text{LMMSE}}(Y)$ is independent of $Y$, and follows a Gaussian distribution with mean zero and covariance matrix as in \eqref{eq:cov_Gaussian}.

\end{prop}
\begin{proof}
The claim that $X\mid Y$ is a Gaussian random vector with mean and covariance matrix as in \eqref{eq:mean_Gaussian} and \eqref{eq:cov_Gaussian} is by \cite[Theorem 3.7 on p.~78]{kaipio2006statistical}. Since the MMSE estimator is equal to the conditional expected value, the identity in \eqref{eq:mean_Gaussian} shows that the MMSE estimator is linear, which proves that the MMSE and the LMMSE estimators coincide. The final assertion follows from the translation invariance of the covariance operator.
\end{proof}

\section{The least squares problem}\label{sec:lstsq}

Throughout this work, we suppose that $n$ independent samples $\{(X_i,Y_i)\}_{i=1}^n$ of $(X,Y)$ are available. To simplify the presentation of our results, we use the following notation:
\begin{enumerate}[label=\arabic*.)]
\item We let $\Theta_* \in \mathbb{R}^{M \times N}$ be the  matrix representing the LMMSE estimator, \ie we define $\Theta_*^\top \coloneqq (A^\top C_{ZZ}^{-1} A + C_{XX}^{-1})^{-1} A^\top C_{ZZ}^{-1}$, such that $f_\text{LMMSE}(Y) = \Theta_*^\top Y$.

\item The estimation error is given by $E = X - \Theta_* Y$. Accordingly, we define the estimation error in the $i$th 
sample by $E_i \coloneqq X_i - \Theta_*^\top Y_i$ for $1 \leq i \leq n$.
\item We group the data into matrices by defining $\boldsymbol{Y}  \coloneqq (Y_1, \dots, Y_n)^\top \in \mathbb{R}^{n \times M}$, $\boldsymbol{X} \coloneqq (X_1, \dots, X_n)^\top \in \mathbb{R}^{n \times N}$, and $\boldsymbol{E} \coloneqq (E_i, \dots, E_n)^\top \in \mathbb{R}^{n \times N} $. 

\end{enumerate}
Simple calculations yield $\boldsymbol{E}= \boldsymbol{X} - \boldsymbol{Y} \Theta_*$ and
\begin{equation}\label{eq:cost_functional}
\frac{1}{n} \sum_{i=1}^n \|\Theta^\top Y_i - X_i\|_2^2 =  \frac{1}{n} \|\boldsymbol{Y} \Theta - \boldsymbol{X} \|_F^2,
\end{equation}
where $\|\cdot \|_F$ denotes the Frobenius norm. The identity in \eqref{eq:cost_functional} implies that the empirical mean squared error minimization problem in \eqref{eq:EMSE} is equivalent to
\begin{equation}\label{eq:EMSE_Theta}
\min_{\Theta \in \mathbb{R}^{M \times N}} (1/n) \|\boldsymbol{Y} \Theta - \boldsymbol{X} \|_F^2.
\end{equation}
If $\boldsymbol{Y}$ is injective, then \eqref{eq:EMSE_Theta}  has a unique solution, which is given by 
\begin{equation}\label{eq:least_squares estimator}
\hat{\Theta}= (\boldsymbol{Y}^\top \boldsymbol{Y})^{-1} \boldsymbol{Y}^\top  \boldsymbol{X}.
\end{equation}
Note that $\hat{\Theta}$ is the matrix representation of the least squares estimator, \ie $\hat{f}(Y) = \hat{\Theta}^\top Y$. The difference of the least squares estimator and the LMMSE estimator is given by
\begin{equation}\label{eq:diff_lstsq_LMMSE_stimator}
\hat{\Theta} - \Theta_*  = (\boldsymbol{Y}^\top \boldsymbol{Y})^{-1} \boldsymbol{Y}^\top \boldsymbol{E}. 
\end{equation}
\section{Expected mean squared error}\label{sec:asymptotic-mse}
A standard trick in linear regression is to use the linearity of the expectation operator to derive results for multi-dimensional responses (in our case the response is $X$) from results for one-dimensional responses.
We use this trick in the derivation of the expected approximation error of the least squares estimator $\hat{\Theta}$. Whenever the response is one-dimensional, we write $\hat{\theta}$ and $\theta_*$ instead of $\hat{\Theta}$ and $\Theta_*$.

\begin{thm}\label{thm:expected_approx_1d}
Let $N=1$ and assume that there is $\sigma \geq 0$ such that $\cov{E \mid Y} \leq \sigma^2$ almost surely. Moreover, suppose that $\boldsymbol{Z} \coloneqq \boldsymbol{Y} C_{YY}^{-1/2}$ is almost surely injective. Then the least squares estimator is almost surely unique and
\begin{equation}\label{eq:exp_approx_error_1d}
\ev{ | \hat{\theta}^\top Y - \theta_*^\top Y |^2 } \leq \sigma^2 \ev{ \trace{ (\boldsymbol{Z}^\top \boldsymbol{Z})^{-1} }} + \ev_{\boldsymbol{Y}}{\|(\boldsymbol{Z}^\top \boldsymbol{Z})^{-1} \boldsymbol{Z}^\top \ev{\boldsymbol{E} \mid \boldsymbol{Y}} \|_2^2 }. \\
\end{equation}
The relation in \eqref{eq:exp_approx_error_1d} is an equality if and only if $\cov{E \mid Y} = \sigma^2$ almost surely.
\end{thm}
\begin{proof}
We calculate
\begin{equation}\label{eq:init_error_decomp}
\begin{split}
\ev{ | \hat{\theta}^\top Y - \theta_*^\top Y |^2 } 
	&= \ev{\boldsymbol{E}^\top \boldsymbol{Y}(\boldsymbol{Y}^\top \boldsymbol{Y})^{-1}  Y Y^\top   		(\boldsymbol{Y}^\top \boldsymbol{Y})^{-1} \boldsymbol{Y}^\top \boldsymbol{E} } \\ 
	&=  \ev_{\boldsymbol{Y}}{ \trace{ \ev_{\boldsymbol{E}}{\boldsymbol{E} \boldsymbol{E}^\top\mid 	    \boldsymbol{Y} }  \boldsymbol{Y} (\boldsymbol{Y}^\top \boldsymbol{Y})^{-1} \ev_Y{Y Y^\top}         	(\boldsymbol{Y}^\top \boldsymbol{Y})^{-1} \boldsymbol{Y}^\top }} \\
		&= \ev_{\boldsymbol{Y}}{\trace{ \ev{\boldsymbol{E} \boldsymbol{E}^\top \mid \boldsymbol{Y} } \boldsymbol{Y} 	   		(\boldsymbol{Y}^\top \boldsymbol{Y})^{-1} C_{YY} (\boldsymbol{Y}^\top \boldsymbol{Y})^{-1} 			\boldsymbol{Y}^\top }} \\
		&= \ev_{\boldsymbol{Y}}{\trace{ \ev{\boldsymbol{E} \boldsymbol{E}^\top \mid \boldsymbol{Y} } \boldsymbol{Z} (\boldsymbol{Z}^\top \boldsymbol{Z})^{-2} \boldsymbol{Z}^\top }} \\		
		&= \ev_{\boldsymbol{Y}}{\trace{ \cov{\boldsymbol{E} \mid \boldsymbol{Y} } \boldsymbol{Z} 	   		(\boldsymbol{Z}^\top \boldsymbol{Z})^{-2} \boldsymbol{Z}^\top }} \\		
		&\quad + \ev_{\boldsymbol{Y}}{\trace{ \ev{\boldsymbol{E}^\top \mid \boldsymbol{Y} } \boldsymbol{Z} (\boldsymbol{Z}^\top \boldsymbol{Z})^{-2} \boldsymbol{Z}^\top \ev{\boldsymbol{E} \mid \boldsymbol{Y} }}}. \\
\end{split}
\end{equation}	
The first identity is by the trace trick and the identity in \eqref{eq:diff_lstsq_LMMSE_stimator}. The second identity is by the law of iterated expectations. The third identity holds because $\ev_Y{Y Y^\top} = C_{YY}$. The fourth identity is true since $\boldsymbol{Y} (\boldsymbol{Y}^\top \boldsymbol{Y})^{-1} C_{YY} (\boldsymbol{Y}^\top \boldsymbol{Y})^{-1} \boldsymbol{Y}^\top = \boldsymbol{Z} (\boldsymbol{Z}^\top \boldsymbol{Z})^{-2} \boldsymbol{Z}^\top$. The fifth identity follows from the identity $\cov{\boldsymbol{E} \mid \boldsymbol{Y}} = \ev{\boldsymbol{E}^\top \boldsymbol{E} \mid \boldsymbol{Y}} - \ev{\boldsymbol{E}^\top \mid \boldsymbol{Y}} \ev{\boldsymbol{E}^\top \mid \boldsymbol{Y}} $ and the cyclic property of the trace. For the second to last expected value in \eqref{eq:init_error_decomp}, we have
	\begin{multline}\label{eq:irr_error_decomp}
	\ev_{\boldsymbol{Y}}{\trace{ \cov{\boldsymbol{E} \mid \boldsymbol{Y} } \boldsymbol{Z} 	   		(\boldsymbol{Z}^\top \boldsymbol{Z})^{-2} \boldsymbol{Z}^\top }}  \leq \ev_{\boldsymbol{Y}}{ \sum_{i=1}^n \lambda_i(\cov{\boldsymbol{E} \mid \boldsymbol{Y} })  		\lambda_i\left( \boldsymbol{Z} (\boldsymbol{Z}^\top \boldsymbol{Z})^{-2} \boldsymbol{Z}^\top \right) } 
	\\ \leq \sigma^2 \ev{ \sum_{i=1}^n  \lambda_i\left(\boldsymbol{Z} (\boldsymbol{Z}^		\top \boldsymbol{Z})^{-2}  \boldsymbol{Z}^\top\right) }  
	= \sigma^2 \ev{ \trace{ (\boldsymbol{Z}^\top \boldsymbol{Z})^{-1} }},
\end{multline} 
where $\lambda_1(B) \leq \cdots \leq \lambda_k(B)$ denote the eigenvalues of a symmetric positive semidefinite matrix $B \in \mathbb{R}^{k \times k}$. The first relation in \eqref{eq:irr_error_decomp} is by von Neumann's trace inequality \cite[Theorem 7.4.1.1 on p. 458]{horn2012matrix}. The second relation holds because $\cov{\boldsymbol{E} \mid \boldsymbol{Y} }$ is a diagonal matrix with diagonal entries bounded by $\sigma^2$. The third relation is valid since the nonzero eigenvalues of $(\boldsymbol{Z}^\top \boldsymbol{Z})^{-1}$ and $\boldsymbol{Z} (\boldsymbol{Z}^\top \boldsymbol{Z})^{-2} \boldsymbol{Z}^\top$ are identical. Using the identity $\trace{V V^\top} = \| V \|_2^2$ for $V \coloneqq (\boldsymbol{Z}^\top \boldsymbol{Z})^{-2} \boldsymbol{Z}^\top \ev{\boldsymbol{E} \mid \boldsymbol{Y}}$, for the last expected value in \eqref{eq:init_error_decomp}, we obtain that 
\begin{equation}\label{eq:miss_error_decomp}
\ev_{\boldsymbol{Y}}{\trace{ \ev{\boldsymbol{E}^\top \mid \boldsymbol{Y} } \boldsymbol{Z} (\boldsymbol{Z}^\top \boldsymbol{Z})^{-2} \boldsymbol{Z}^\top \ev{\boldsymbol{E} \mid \boldsymbol{Y} }}} = \ev_{\boldsymbol{Y}}{\|(\boldsymbol{Z}^\top \boldsymbol{Z})^{-1} \boldsymbol{Z}^\top \ev{\boldsymbol{E} \mid \boldsymbol{Y}} \|_2^2 }.
\end{equation}
Together, \crefrange{eq:init_error_decomp}{eq:miss_error_decomp} yield \eqref{eq:exp_approx_error_1d}. To complete the proof, we note that the first and second relations in \eqref{eq:irr_error_decomp} are equalities if and only if 
$\cov{\boldsymbol{E} \mid \boldsymbol{Y} } = \sigma^2 \id$, and that the latter condition is equivalent to the condition that $\cov{E \mid Y} = \sigma^2$ almost surely.
\end{proof}

We write $G \preceq H$ for two matrices $G, H \in \mathbb{R}^{N \times N}$ if $H-G$ is positive semidefinite. 

\begin{thm}[Expected approximation error]\label{prop:expected_out_of_sample}
Let $C \in \mathbb{R}^{N \times N}$ be such that $\cov{E \mid Y} \preceq C$ almost surely. Moreover, suppose that $\boldsymbol{Z} \coloneqq \boldsymbol{Y} C_{YY}^{-1/2}$ is almost surely injective. Then the least squares estimator is almost surely unique and
\begin{equation}\label{eq:exp_approx_error}
\ev{ \| \hat{\Theta}^\top Y - \Theta_*^\top Y \|_2^2 } \leq \trace{C} \ev{ \trace{ (\boldsymbol{Z}^\top \boldsymbol{Z})^{-1} }} + \ev_{\boldsymbol{Y}}{\|(\boldsymbol{Z}^\top \boldsymbol{Z})^{-1} \boldsymbol{Z}^\top \ev{\boldsymbol{E} \mid \boldsymbol{Y}} \|_F^2 }. \\
\end{equation}
The relation in \eqref{eq:exp_approx_error} is an equality if and only if $\diag(\cov{E \mid Y}) = \diag(C)$ almost surely.
\end{thm}
\begin{proof}
We let $\hat{\theta}^i$ and $\theta_*^i$ denote the $i$th row of $\hat{\Theta}$ and $\Theta_*$ for $1 \leq i \leq N$. Then 
\begin{equation}\label{eq:error_sum_comp}
\ev{ \| \hat{\Theta}^\top Y - \Theta_*^\top Y \|_2^2 } = \ev{ \sum_{i=1}^N | {\hat{\theta}^i}^\top Y - {\theta_*^i}^\top Y |_2^2 }.
\end{equation}
The claim follows by using the linearity of the expectation operator, and by applying \Cref{thm:expected_approx_1d} for each summand on the right-hand side of \eqref{eq:error_sum_comp}. 
\end{proof}
\Cref{prop:expected_out_of_sample} shows that the expected approximation error of the least squares estimator can be decomposed into two parts: the first part depends on the mean squared error of the MMSE estimator. The second part depends on the distance of the LMMSE to the MMSE estimators. Notice the similarity of the error decomposition in \eqref{eq:exp_approx_error} to the error decomposition in \eqref{eq:pythagoras}.

\begin{rem}[Expected mean squared error for the Gaussian model]\label{rem:expected_error_Gaussian_model}
The consequences of \Cref{prop:expected_out_of_sample} for the Gaussian model are well-known. 
For the Gaussian model, by \Cref{prop:posterior_gaussian_model}, we have $\ev{E \mid Y} = 0$; hence, the second expected value in the error decomposition \eqref{eq:exp_approx_error} vanishes. Moreover, by \Cref{prop:posterior_gaussian_model}, we can choose $C = C_{EE}$. We observe that $\boldsymbol{Z}^\top \boldsymbol{Z}$ has a Wishart distribution with the identity as a scale matrix and degrees of freedom parameter $n$ (see \cite[Definition 3.4.1 on p.~66]{Mardia1979}). Thus, by \cite[Lemma 7.7.1 on p.~273]{anderson2003introduction}, if $n > M+1$, then $\ev{(\boldsymbol{Z}^\top \boldsymbol{Z})^{-1}} = (n-M-1)^{-1} \id $, where $\id$ is the $M \times M$ identity matrix. Since the trace operator is linear and the trace of the $M \times M$ identity matrix is $M$, it follows that
$
\ev{ \trace{ (\boldsymbol{Z}^\top \boldsymbol{Z})^{-1} }} = M/(n-M-1).
$
Hence,
\begin{equation}\label{eq:approx_error_gauss}
\ev{ \| \hat{\Theta}^\top Y - \Theta_*^\top Y \|_2^2 } = \trace{C_{EE}} M/(n-M-1).
\end{equation}
By combining the identity in \eqref{eq:approx_error_gauss} with the error decomposition in \eqref{MSE_decompostion_linear}, we deduce that for the Gaussian model the answer to \Cref{quest:avg_err} is $n=\lceil{M/\varepsilon\rceil} + M + 1$.

\end{rem}

The inverse of the matrix $\boldsymbol{Z}^\top \boldsymbol{Z}$ plays an important role in \Cref{prop:expected_out_of_sample}. In \Cref{rem:expected_error_Gaussian_model}, we have seen that for the Gaussian model the expected value of its trace is equal to $M/(n-M-1)$. Next, we derive a lower bound for the expected value of its trace for the general model. This lower bound was already used in the proof of \cite[Theorem 1]{rosset2019fixed}.
\begin{prop}
Assume that the random vector $\zeta$ takes values in $\mathbb{R}^M$, has mean zero and the identity as a covariance matrix, and has a density with respect to the Lebesgue measure. Let $n \geq M$ and assume that the rows of $\boldsymbol{Z} \in \mathbb{R}^{n \times M}$ are independent copies of $\zeta$. Then $\boldsymbol{Z}^\top \boldsymbol{Z}$ is almost surely invertible and 
\begin{equation}\label{eq:lower_bound_n_samples}
\ev{ \trace{ (\boldsymbol{Z}^\top \boldsymbol{Z})^{-1} }} \geq M/n.
\end{equation}
\end{prop}
\begin{proof}
The invertibility of $\boldsymbol{Z}^\top \boldsymbol{Z}$ follows from \cite[Theorem 2.3]{eaton1973non}. If the expected value in \eqref{eq:lower_bound_n_samples} is infinite, the inequality in  \eqref{eq:lower_bound_n_samples} is trivially satisfied; otherwise, \begin{equation*}
\ev{ \trace{ (\boldsymbol{Z}^\top \boldsymbol{Z})^{-1} }} =  \trace{\ev{  (\boldsymbol{Z}^\top \boldsymbol{Z})^{-1} }} \geq \trace{\ev{  (\boldsymbol{Z}^\top \boldsymbol{Z}) }^{-1}} = \trace{(1/n) \id} = M/n.  
\end{equation*}
The first identity follows from the linearity of the expectation operator. The second relation holds because $\ev{(\boldsymbol{Z}^\top \boldsymbol{Z})^{-1}} -\ev{\boldsymbol{Z}^\top \boldsymbol{Z}}^{-1}$ is positive semidefinite; see \cite{groves1969}. The third identity is true because $\ev{\boldsymbol{Z}^\top \boldsymbol{Z}} = n \id$. The final identity holds because the trace of the $M \times M$ identity matrix is $M$.
\end{proof}
We now recall that under appropriate assumptions, and for $n=n(M)$ such that $\lim_{M \to \infty} M/n(M) = \gamma$ for some $\gamma \in (0,1)$, the trace of the inverse of $\boldsymbol{Z}^\top \boldsymbol{Z}$ converges almost surely to the expected value for the Gaussian model (see \Cref{rem:expected_error_Gaussian_model}) as $M$ tends to $\infty$. We refer to \cite[Theorem 2 on p.~301 and Corollary 4 on p.~303]{shiryaev2016}) for the intrinsic definition of a probability space on a sequence of random variables, where each random variable is defined on its own probability space.

\begin{thm}\label{thm:conc_trace_prec}
Let $q$ be a random variable with mean zero, unit variance, and finite fourth moment. For each $M \in \mathbb{N}$, let $\zeta(M)$ be a random vector taking values in $\mathbb{R}^M$ whose entries are independent copies of $q$. Let $n(M)$ be such that $\lim_{M \to \infty} M/n(M) = \gamma$ for some $\gamma \in (0,1)$. Let $\boldsymbol{Z}=\boldsymbol{Z}(n(M),M)$ be a random matrix in $\mathbb{R}^{n(M) \times M}$ whose rows are independent copies of $\zeta(M)$. Then
\begin{equation*}
\lim_{M \to \infty}  \trace{ (\boldsymbol{Z}^\top \boldsymbol{Z})^{-1} } = \lim_{M \to \infty}  M/(n(M)-M-1) = \gamma/(1-\gamma) \quad \text{almost surely.}
\end{equation*}
\end{thm}
\begin{proof}
See \cite[Proposition 2 and its proof]{hastie2019surprises}. 
\end{proof}

The independence assumption on the entries of $\bold{Z}$ in \Cref{thm:conc_trace_prec} is restrictive, however, we expect that it can be weakened to include particular cases of interest; see \Cref{rem:marchenko-pastur}. In \Cref{prop:Gaussian_error_converges_in_p}, we combine \Cref{thm:conc_trace_prec} with a non-asymptotic two-sided tail bound to show that for $n = n(M)$ such that $\lim_{M \to \infty} M/n(M) = \gamma$ the approximation error for the Gaussian model converges to $\trace{C_{EE}} \gamma/(1-\gamma)$ in probability as $M \to \infty$.

\begin{rem}\label{rem:marchenko-pastur}
The assumptions about the entries of $\bold{Z}$ in \Cref{thm:conc_trace_prec} hold for the Gaussian model, since there the entries of $\bold{Z}$ are independent standard normal random variables. In general, however, the independence assumption on the columns of $\bold{Z}$ is restrictive. The proof of \Cref{thm:conc_trace_prec} in \cite[Proposition 2]{hastie2019surprises} consists of a combination of two results: the Marchenko--Pastur Law \cite{marvcenko1967distribution} and the Bai--Yin theorem \cite{bai1993limit}. Since there exist variations of both of these results that hold under assumptions that are weaker than the assumptions in \Cref{thm:conc_trace_prec}---see \eg \cite[Theorem 2.1]{yaskov2016necessary}, and \cite[Corollary 3.1]{yaskov2016controlling} and \cite{chafai2018convergence}---we expect that the assumptions in \Cref{thm:conc_trace_prec} can be  weakened. We are currently working on an extension of \Cref{thm:conc_trace_prec} which uses the typical structure of $\bold{Z}$ in an inverse problem with additive noise. An extension of the Marchenko--Pastur law to random matrices with independent rows and blockwise independent columns is provided in \cite{bryson2019marchenko}. 

\end{rem}

\section{Tail bounds for the mean squared error}\label{sec:non-asymptotic-mse}
We derive non-asymptotic tail bounds for the approximation error under sub-Gaussian conditions. We decided to work with sub-Gaussian variables for three reasons: first, the fact that Gaussians are sub-Gaussian (\Cref{rem:sub-Gaussian-random-variables} \ref{item:gaussian_sub_gaussian}) allows us to compare our results with the benchmark in \eqref{eq:formula_n_gaussian}. Second, the class of sub-Gaussian variables is reasonably general since it contains, for example, all bounded random variables (\Cref{rem:sub-Gaussian-random-variables} \ref{item:hoeffding_lemma}). Third, sub-Gaussian variables are convenient to work with since, as we will see, many powerful results about sub-Gaussian variables are available in the literature.
\begin{rem}
For $N=1$, the distribution of the approximation error for the Gaussian model is known: the rescaled approximation error $n \sigma^{-2} \ev_Y{ |\hat{\theta}^\top Y - \theta_*^\top Y |^2 }$ follows Hotelling's $T$-squared distribution with dimensionality parameter $M$ and degrees of freedom $n$, where $\sigma^2 \coloneqq \cov{E}$;  see \cite[Theorem~1.3 and its proof]{breiman1983many}. In this case, exact confidence intervals for the error can be obtained by using the fact that if $V$ is distributed according to Hotelling's $T$-squared distribution with dimensionality parameter $M$ and degrees of freedom $n$, then $V (n-M+1)/(Mn) $ follows an F-distribution (see \cite[Theorem 3.5.2 on p.~74]{Mardia1979}).
\end{rem}

\subsection{Sub-Gaussian random variables}
We recall the definition and some properties of sub-Gaussian variables. A real-valued random variable $V$ is called {sub-Gaussian} if there exists $\sigma \geq 0$ such that
\begin{equation}\label{eq:sub-gaussian_condition}
\ev{e^{\lambda(V - \mu)}} \leq e^{\frac{\sigma^2 \lambda^2}{2}} \quad \text{for all }\lambda \in \mathbb{R}.
\end{equation}
The constant $\sigma$ in \eqref{eq:sub-gaussian_condition} is called a sub-Gaussian parameter of $V$. A random variable $V$ that is sub-Gaussian with parameter $\sigma$ satisfies, for all $t \geq 0$, the concentration inequality $ 
\mathbb{P}(|V - \ev{V}| \geq  t ) \leq  2 \exp\left(-t^2/(2 \sigma^2)\right)
$;
see \cite[p.~23]{wainwright2019high}.
\begin{rem}\label{rem:sub-Gaussian-random-variables}  The following properties hold: \hfill
\begin{enumerate}[label=\roman*)]
\item If a random variable $V$ is sub-Gaussian with parameter $\sigma$, then $\cov{V} \leq \sigma^2$; see \cite[Lemma~1.2 on p.~3]{buldygin2000metric}. \label{item:cov_smaller_sub-Gaussian}
\item A Gaussian random variable with variance $\sigma^2$ is sub-Gaussian with parameter $\sigma$; see \cite[Example~2.1 on p.~22]{wainwright2019high}. \label{item:gaussian_sub_gaussian}
\item A random variable $V$ that is uniformly distributed on the interval $[a,b]$, $a<b$, is sub-Gaussian with parameter $\sigma = (b-a)/\sqrt{12}$; see \cite[Section~4.3]{arbel2020strict}.\label{item:unif_dist_strictly_sub}

\item A random variable that is supported in the intervall $[a,b]$ is sub-Gaussian with parameter $\sigma = (b-a)/2$; see \cite[Lemma~2.6 on p.~21]{massart2007concentration}.\label{item:hoeffding_lemma}
\end{enumerate}
\end{rem}
We now extend the definition of sub-Gaussians to random vectors (see \cite[p.~165]{wainwright2019high}).

\begin{defn}[Sub-Gaussian random vectors]
A random vector $V$ that takes values in $\mathbb{R}^N$ is called a {$\sigma$-sub-Gaussian random vector} if $c^\top V$ is a sub-Gaussian random variable with parameter $\sigma$ for every $c \in \mathbb{R}^N$ with $\| c \|_2=1$. 
\end{defn}
It is easy to see that $V$ is a $\sigma$-sub-Gaussian random vector if and only if 
\begin{equation*}
\ev{ \exp(\alpha^\top (V-\ev{V})) } \leq \exp\left(\|\alpha\|_2^2 \sigma^2/2\right) \quad \text{for all } \alpha \in \mathbb{R}^n. 
\end{equation*}
Moreover, a simple calculation shows that if the entries of $V$ are independent sub-Gaussian random variables with parameter $\sigma$, then $V$ is a $\sigma$-sub-Gaussian random vector. In combination with \Cref{rem:sub-Gaussian-random-variables} \ref{item:gaussian_sub_gaussian}, this implies that a random vector whose entries are independent standard normal random variables is $1$-sub-Gaussian. Next, we recall a tail bound for quadratic forms of sub-Gaussian random vectors from \cite[Theorem 2.1]{hsu2012tail}. Thereby, $\| \cdot \|_2$ returns the spectral norm if evaluated for matrices, and the Euclidean norm otherwise.

\begin{thm}[Quadratic forms of sub-Gaussian vectors]\label{thm:tail_quadratic_subgaussians}
Let $H \in \mathbb{R}^{m \times n}$ and define $\Sigma \coloneqq H^\top H$. Suppose that $V \in \mathbb{R}^n$ is a $\sigma$-sub-Gaussian random vector. Then, for all $t \geq 0$, 
\begin{equation*}
\| H V \|_2^2 \leq \sigma^2 \left( \trace{\Sigma} + 2 \sqrt{\trace{\Sigma^2}t} + 2 \| \Sigma \|_2 t \right) + \| H \ev{V} \|_2^2 \left(1+2 \left(\|\Sigma\|_2^2 t /\trace{\Sigma^2} \right)^{1/2} \right) 
\end{equation*}
with probability at least $1-e^{-t}$.
\end{thm}
\Cref{thm:tail_quadratic_subgaussians} is similar to the Hanson-Wright inequality \cite{rudelson2013hanson}, which provides two-sided tail bounds for quadratic forms of sub-Gaussian vectors with independent entries. Moreover, it is related to the two-sided tail bound for quadratic forms of Gaussian random variables from \cite[Lemma 1]{laurent2000adaptive}. As a first application of \Cref{thm:tail_quadratic_subgaussians}, we derive tail bounds for the norm of sums of independent zero-mean sub-Gaussian random vectors.
\begin{lem}[Sums of independent sub-Gaussian random vectors]\label{lem:sums-of-unbounded}
Let $V_1, \dots, V_n$ be independent zero-mean $\sigma$-sub-Gaussian random vectors that take values in $\mathbb{R}^M$.  Then $\mathbb{P}\left((1/n)\|\sum_{i=1}^n V_i \|_2^2 \geq \sigma^2 (M + 2 \sqrt{M  t} + 2 t \right) \leq e^{-t}.$
\end{lem}
\begin{proof}
Our first aim is to prove that $\vec{V} \coloneqq (V_1^\top, \dots, V_{n}^\top)^\top$ is a $\sigma$-sub-Gaussian random vector in $\mathbb{R}^{nM}$. For this purpose, we let $\alpha_1, \dots, \alpha_n \in \mathbb{R}^M$, define $\alpha \coloneqq (\alpha_1^\top, \dots, \alpha_n^\top)^\top \in \mathbb{R}^{nM}$, and estimate
\begin{equation}\label{eq:v_subgaussian}
\ev{\exp(\alpha^\top \vec{V})} =  \prod_{i=1}^n \ev{ \exp(\alpha_i^\top V_i) }  \leq \prod_{i=1}^n \exp(\|\alpha_i\|_2^2 \sigma^2 / 2 )  = \exp(\|\alpha\|_2^2 \sigma^2 /2).
\end{equation}
The first identity holds because $V_1, \dots, V_n$ are independent. The second relation is true since $V_1, \dots, V_n$ are $\sigma$-sub-Gaussian random vectors. The estimate in \eqref{eq:v_subgaussian} implies that $\vec{V}$ is a $\sigma$-sub-Gaussian random vector. Our next aim is to show that  $\|\sum_{i=1}^n V_i \|_2^2$ is equal to a quadratic form in $\vec{V}$. For this purpose, we define the matrix $B=(B_{i,j}) \in \mathbb{R}^{nM \times nM}$ by
\begin{equation*}
B_{i,j} \coloneqq \begin{cases} 1 \quad \text{if there exists } k \in \mathbb{Z} \text{ such that } j - i = k M , \\ 
0 \quad \text{otherwise.}
\end{cases} 
\end{equation*}
It is easy to see that $B$ is symmetric. Let $V_{i,j}$ denote the $j$th entry of $V_i$. Then
\begin{equation*}
\vec{V}^\top B \vec{V} = \sum_{i=1}^n \sum_{j=1}^n \sum_{k=1}^M \sum_{l=1}^M B_{k+(i-1)M, l+(j-1)M} V_{i,k} V_{j,l}  = \sum_{i=1}^n \sum_{j=1}^n \sum_{k=1}^M V_{i,k} V_{j,k} = \|\sum_{i=1}^n V_i \|_2^2.
\end{equation*}
This proves that $\|\sum_{i=1}^n V_i \|_2^2 = \vec{V}^\top B \vec{V}$ and that $B$ is positive semidefinite. Moreover,  we have $\trace{B} = nM$, $\trace{B^2} = \| B \|_F^2 = n^2 M$,  
and
$
\| \sum_{i=1}^n  V_i\|_2^2 \leq n \sum_{i=1}^n \| V_i\|_2^2,  
$
Since the last relation is an equality for $V_1 = \cdots = V_n$, we have $\| B \|_2 = n$. The proof is completed by applying \Cref{thm:tail_quadratic_subgaussians} to $H = B^{1/2}$ and $V = \vec{V}$,  where $B^{1/2}$ is a square root of $B$. 
\end{proof}
The product of a sub-Gaussian and a bounded random variable is sub-Gaussian.
\begin{lem}[Products of sub-Gaussian and bounded variables]\label{lem:products_of_sub_gaussian_rv}
Let $V$ and $W$ be random variables such that $\ev{V} = \ev{VW} = 0$, $V$ is sub-Gaussian with parameter $\sigma$, and $|W|$ is almost surely bounded by $b \geq 0$. Then $VW$ is sub-Gaussian with parameter $b \sigma \sqrt{3/2}$.
\end{lem}
\begin{proof}
For all $t \in \mathbb{R}$, we have  
\begin{equation}\label{eq:cosh_estimate}
\exp(t) \leq \cosh\left(t \sqrt{3/2} \right) + t.   
\end{equation}
Consequently, for all $\lambda \in \mathbb{R}$ we have 
\begin{equation*}
\ev{\exp(\lambda V W)} \leq \ev{\cosh\Big(\lambda b V \sqrt{3/2} \Big)} + \ev{\lambda V W} \leq \exp(\lambda^2 b^2  \sigma^2 3/4). 
\end{equation*}
The first relation follows from the estimate in \eqref{eq:cosh_estimate} and the fact that  
$\cosh(r) \leq \cosh(s)$ for $r,s \in \mathbb{R}$ with $|r| \leq |s|$. The second relation follows by applying the defining property of sub-Gaussians in \eqref{eq:sub-gaussian_condition} twice, once with $\lambda b \sqrt{3/2}$ and once with $-\lambda b \sqrt{3/2}$ in place of $\lambda$, and by using the assumption that $\ev{VW}=0$. The proof is complete. 
\end{proof}
By combining \Cref{lem:sums-of-unbounded} and \Cref{lem:products_of_sub_gaussian_rv} we obtain the following corollary.
\begin{cor}[Weighted sums of independent sub-Gaussian vectors]\label{cor:weighted_sum_sub-gaussians} Let $\{(w_i, V_i)\}_{i=1}^n$ be independent copies of a bounded random variable $w$ und a $\sigma$-sub-Gaussian random vector $V$ that takes values in $\mathbb{R}^M$. Suppose that $\ev{V} = \ev{wV} = 0$, and that $| w | \leq b $ almost surely for $b \geq 0$. Then $\mathbb{P}\left((1/n)\|\sum_{i=1}^n w_i V_i \|_2^2 \geq (3/2) \sigma^2 b^2 (M + 2 \sqrt{M t} + 2 t \right) \leq e^{-t}.$

\end{cor}
We now extend \Cref{thm:tail_quadratic_subgaussians} to matrices with independent sub-Gaussian rows.

\begin{lem}[Quadratic form of random matrices with independent sub-Gaussian rows]\label{lem:tail_inequality_quadratic_max_forms}
Let $H \in \mathbb{R}^{k \times n}$ be a matrix, and define $\Sigma \coloneqq H^\top H$. Let $C$ be a symmetric positive-definite matrix in $\mathbb{R}^{N \times N }$. Suppose that $E$ is a $\mathbb{R}^{n \times N}$ random matrix with independent rows $E_i$ that satisfy
\begin{equation}\label{eq:companion_matrix}
\ev{ \exp \left( \alpha^\top C^{-1/2} (E_i-\ev{E_i}) \right)} \leq \exp\left(\| \alpha \|_2^2 / 2 \right) \quad \text{for all } \alpha \in \mathbb{R}^N,
\end{equation}
for $1 \leq i \leq n$. Then, with probability at least $1-e^{-t}$,
\begin{multline*}
\| H E \|_F^2 \leq \trace{C} \trace{\Sigma} + 2 \sqrt{\trace{C^2} \trace{\Sigma^2} t}  + 2 \| C \|_2 \| \Sigma \|_2t  \\ + \| H \ev{E} \|_F^2    \left( 1 +  2 \left(\frac{\| C \|_2^2 \| \Sigma \|_2^2}{\trace{C^2} \trace{\Sigma^2} } t\right)^{1/2} \right). 
\end{multline*}
\end{lem}
\begin{proof}
We factorize $C$ as $C = U D U^\top$, where $U \in \mathbb{R}^{N \times N}$ is an  orthogonal matrix and $D=\diag(d_1, \dots, d_N) \in \mathbb{R}^{N \times N}$ is a diagonal matrix. For $F = E C^{-1/2}$, we have 
\begin{equation*}
\| H E \|_F^2 = \| H F U D^{1/2} U^\top \|_F^2 = \| H F U D^{1/2} \|_F^2.
\end{equation*}
We define $G \coloneqq F U$. The rows $G_i$ of $G$ are independent and satisfy, for all $\alpha \in \mathbb{R}^N$,  
\begin{equation}\label{eq:G_subgaussian}
\ev{ \exp \left( \alpha^\top G_i \right)} = \ev{ \exp \left( \alpha^\top U^\top C^{-1/2}E_i \right)} \leq \exp\left(\| U \alpha \|_2^2 / 2 \right) = \exp\left(\| \alpha \|_2^2 / 2 \right).
\end{equation}
For the block diagonal matrix $S \coloneqq \text{diag}\left( d_1^{1/2} H, \dots, d_N^{1/2} H \right)\in \mathbb{R}^{kN \times nN}$ and $\vec{G} \in \mathbb{R}^{nN}$ defined by
$
\vec{G} \coloneqq \left(G_{1,1}, \dots, G_{n,1}, G_{1,2}, \dots, G_{n,2}, \dots, G_{1,N}, \dots, G_{n,N} \right)^\top
$, we have
\begin{equation*}
\| H G D^{1/2} \|_F^2 = \| S \vec{G} \|_2^2.
\end{equation*}
Let $\alpha \in \mathbb{R}^{nN}$, and define $\alpha^i = \left(\alpha_{0+i}, \alpha_{n+i}, \dots, \alpha_{(N-1)n + i} \right)^\top \in \mathbb{R}^N$ for $1 \leq i \leq n$ . Then 
\begin{multline}\label{eq:vec_G_sub-gausian}
\ev{ \exp \left( \alpha^\top \vec{G} \right) } = \ev{ \exp \left(\sum_{i=1}^n \sum_{j=1}^N \alpha_{(j-1)n+i} G_{i,j} \right) } = \ev{ \prod_{i=1}^n \exp\left( \sum_{j=1}^N \alpha_{(j-1)n+i} G_{i,j} \right)} \\  = \ev{ \prod_{i=1}^n \exp\left( (\alpha^i)^\top G_i \right)} = \prod_{i=1}^n \ev{  \exp\left( (\alpha^i)^\top G_i \right)}  \leq \prod_{i=1}^n \exp\left(\| \alpha^i \|_2^2/2\right)  =  \exp\left(\| \alpha\|_2^2/2\right).
\end{multline}
The first identity follows by definition of $\vec{G}$. The second identity holds because the exponential map converts a sum of real scalars to a product. The third identity holds by definition of $\alpha^i$. The fourth identity follows from the independence of the rows of $G$. The fifth relation follows from \eqref{eq:G_subgaussian}. The sixth identity holds because the logarithm converts a product of positive scalars to a sum. Since the relation in \eqref{eq:vec_G_sub-gausian} shows that $\vec{G}$ is a 1-sub-Gaussian random vector, we can apply \Cref{thm:tail_quadratic_subgaussians} to $H=S$ and $V = \vec{G}$ to obtain
\begin{multline}\label{eq:vectorized_quadratic-form_bound}
\| S \vec{G} \|_2^2 \leq \trace{S^\top S} + 2 \sqrt{\trace{(S^\top S)^2}t}   + 2 \|S^\top S\|_2 t    + \| S \ev{\vec{G}} \|_2^2    \left( 1 +  2 \left(\frac{ \| S^\top S \|_2^2 t}{\trace{(S^\top S)^2} }\right)^{1/2} \right)   
\end{multline}
with probability at least $1-e^{-t}$. We now need four identities for the terms in \eqref{eq:vectorized_quadratic-form_bound} that can be derived using \cite[Theorem 1.3.1 on p.~23]{golub2013matrix}: the first identity is 
\begin{equation*}
\trace{S^\top S} = \sum_{j=1}^N d_j \trace{H^\top H } = \trace{C} \trace{\Sigma}.
\end{equation*}
The second identity is $\trace{(S^\top S)^2} = \trace{C^2} \trace{\Sigma^2}$. The third identity is $\| S^\top S \|_2 = \|C \|_2 \| \Sigma \|_2$. The fourth identity is
$
\|S \ev{\vec{G}} \|_2^2 = \| H \ev{E} \|_F^2.
$
To complete the proof, it suffices to combine these four identities with the estimate in \eqref{eq:vectorized_quadratic-form_bound}
\end{proof}
A random vector $V$ is called isotropic if $\cov{V} = \id$; see \cite[Definition 3.2.1 on p.~47]{vershynin2018high} The following theorem provides tail bounds for the smallest and largest singular values of a random matrix with independent isotropic and sub-Gaussian rows.
\begin{thm}[Singular values of random matrices with independent sub-Gaussian rows]\label{thm:sub_gaussian-rows}
There exists an absolute constant $c>0$ such that for every $n \times M$ matrix $\boldsymbol{Z}$ whose rows are independent $\rho$-sub-Gaussian isotropic random vectors in $\mathbb{R}^M$ and every $t \geq 0$, with probability at least $1- 2 e^{-t}$, we have
\begin{equation}\label{eq:singular_value_estimate}
\sqrt{n}- c \rho^2 (\sqrt{M} + \sqrt{t}) \leq \sigma_{\min}(\boldsymbol{Z}) \leq \sigma_{\max}(\boldsymbol{Z}) \leq \sqrt{n} + c \rho^2(\sqrt{M} + \sqrt{t}),
\end{equation}
where $\sigma_{\min}(\bold{Z})$ and $\sigma_{\max}(\bold{Z})$ denote the minimal and maximal singular values of $\bold{Z}$. If the entries of $\boldsymbol{Z}$ are independent standard normal random variables, then the relations in \eqref{eq:singular_value_estimate} hold with $c=1$ and $\rho=1$.
\end{thm}
\begin{proof}
See \cite[Theorem 4.61 on p.~98]{vershynin2018high} and \cite[Theorem 6.1 on p.~161]{wainwright2019high}.
\end{proof}

We now establish a few matrix identities needed for the proofs of our main results.
\begin{lem}\label{lem:trace_and_trace_square_inequalities}
Let $G \in \mathbb{R}^{n \times m}$ be injective. Then $G^\top G \in \mathbb{R}^{m \times m}$ is invertible and for $H = (G^\top G)^{-1} G^\top \in \mathbb{R}^{m \times n}$ and $\Sigma = H^\top H \in \mathbb{R}^{n \times n}$, we have 
\begin{align}
\trace{\Sigma} = \trace{(G^\top G)^{-1}}, \quad \label{eq:trace_sigma_1} 
\trace{\Sigma^2} = \trace{(G^\top G)^{-2}}, \quad \text{and } 
\| \Sigma \|_2 = \| (G^\top G)^{-1} \|_2.
\end{align}

\end{lem}
\begin{proof}
Since $G$ is injective, we have $x^\top G^\top G x = \| G x \|_2^2  > 0$ for $x \in \mathbb{R}^m \setminus \{0\}$; hence the matrix $G^\top G$ is positive-definite and thus invertible. Using the fact that $\trace{BC} = \trace{CB}$ for $B \in \mathbb{R}^{n \times m}$ and $C \in \mathbb{R}^{m \times n}$, we have  
\begin{align*}
\trace{\Sigma} &= \trace{G (G^\top G)^{-2} G^\top} = \trace{G^\top G (G^\top G)^{-2}} = \trace{(G^\top G)^{-1}}, \\ 
\trace{\Sigma^2} &= \trace{G (G^\top G)^{-2} G^\top G (G^\top G)^{-2} G^\top}, = \trace{G^\top G (G^\top G)^{-3}} = \trace{(G^\top G)^{-2}},
\end{align*}
which proves the first two identities in \eqref{eq:trace_sigma_1}. By \cite[Satz~V.5.2 (f) on p.~237]{Werner2007}, we have $\| H^\top H \|_2 = \| H H^\top \|_2$. In combination with the fact that  
\begin{equation*}
\| H H^\top \|_2 = \|(G^\top G)^{-1} G^\top G (G^\top G)^{-1}\|_2 = \| (G^\top G)^{-1} \|_2,
\end{equation*}
this proves the third identity in \eqref{eq:trace_sigma_1}. The proof is complete.
\end{proof}

\subsection{Main results}
We begin by investigating the conditional approximation error. 
\begin{prop}[Conditional approximation error]\label{prop:cond_out_of_sample_miss}
Suppose that $\boldsymbol{Z} \coloneqq \boldsymbol{Y} C_{YY}^{-1/2}$ is injective. Moreover, assume that there exists a symmetric positive-definite matrix $C \in \mathbb{R}^{N \times N}$ such that $C^{-1/2} E_i$ is a $1$-sub-Gaussian random vector for $1 \leq i \leq n$. Then, for all $t \geq 0$,
\begin{equation}\label{eq:conditioned_bound}
\ev_Y{ \| \hat{\Theta}^\top Y - \Theta_*^\top Y \|_2^2 \mid \boldsymbol{Y}} \leq \varepsilon_{\text{err}}^{\boldsymbol{Y}} + \varepsilon_{\text{bias}}^{\boldsymbol{Y}},
\end{equation}
with probability at least $1-e^{-t}$, where 
\begin{align}
\varepsilon_{\text{err}}^{\boldsymbol{Y}}&\coloneqq \trace{C} \trace{(\boldsymbol{Z}^\top \boldsymbol{Z})^{-1}} + 2 \sqrt{\trace{C^2}\trace{(\boldsymbol{Z}^\top \boldsymbol{Z})^{-2}}t }  + 2 \|C\|_2 \|(\boldsymbol{Z}^\top \boldsymbol{Z})^{-1}\|_2 t, 
 \label{eq:eps_err}\\ 
\varepsilon_{\text{bias}}^{\boldsymbol{Y}} &\coloneqq  \| (\boldsymbol{Z}^\top \boldsymbol{Z})^{-1} \boldsymbol{Z}^\top \ev{\boldsymbol{E} \mid \boldsymbol{Y}} \|_F^2     \left( 1 +  2 \left(\frac{\| C \|_2^2 \| (\boldsymbol{Z}^\top \boldsymbol{Z})^{-1} \|_2^2t}{\trace{C^2} \trace{(\boldsymbol{Z}^\top \boldsymbol{Z})^{-2} }}\right)^{1/2} \right). \label{eq:eps_bias}
\end{align}

\end{prop}
\begin{proof}
We have 
\begin{equation*}
\begin{split}
\ev_Y{ \| \hat{\Theta}^\top Y - \Theta_*^\top Y\|_2^2 \mid \boldsymbol{Y}, \boldsymbol{E}}  &= \trace{ (\hat{\Theta}-\Theta_*)^\top C_{YY} (\hat{\Theta}-\Theta_*) } \\
 &= \trace{\boldsymbol{E}^\top \boldsymbol{Y} (\boldsymbol{Y}^\top \boldsymbol{Y})^{-1} C_{YY} (\boldsymbol{Y}^\top \boldsymbol{Y})^{-1} \boldsymbol{Y}^\top \boldsymbol{E} } \\
  &= \trace{\boldsymbol{E}^\top \boldsymbol{Z} (\boldsymbol{Z}^\top \boldsymbol{Z})^{-2} \boldsymbol{Z}^\top \boldsymbol{E} } \\
&= \| (\boldsymbol{Z}^\top \boldsymbol{Z})^{-1} \boldsymbol{Z}^\top \boldsymbol{E} \|_F^2.
\end{split}
\end{equation*}
The first identity results from the trace trick and the identity $\ev{YY^\top} = C_{YY}$. The second identity follows by the expression for $\hat{\Theta}-\Theta_*$ in \eqref{eq:diff_lstsq_LMMSE_stimator}. The third identity is by the definition of $\bold{Z}$. The fourth identity holds because $\trace{B B^\top} = \| B \|_F^2$ for $B \in \mathbb{R}^{M \times N}$.
To complete the proof, we first apply the bound on quadratic forms of random matrices from \Cref{lem:tail_inequality_quadratic_max_forms} to $H= (\boldsymbol{Z}^\top \boldsymbol{Z})^{-1} \boldsymbol{Z}^\top$ and $\boldsymbol{E} = E$, and then apply \Cref{lem:trace_and_trace_square_inequalities} to $G = \boldsymbol{Z}$. 
\end{proof}

Before we derive our main result, we study the asymptotic behavior of the error bound in \eqref{eq:conditioned_bound} for two special cases.

\begin{prop}\label{rem:asymptotic_err}
Suppose that, for every $M \in \mathbb{N}$, the hypotheses in \Cref{prop:cond_out_of_sample_miss} hold with $C \in \mathbb{R}^{N \times N}$ independent of $M$, and $\ev{E \mid Y} = 0$ almost surely. Moreover, assume that $\boldsymbol{Z}=\boldsymbol{Z}(n(M),M)$ satisfies the hypotheses in \Cref{thm:conc_trace_prec}. Let $n=n(M)$ be such that $\lim_{M \to \infty} M/n(M) = \gamma$ for some $\gamma \in (0,1)$. Then
\begin{equation}\label{eq:err_conv_prob}
\lim_{M \to \infty} \mathbb{P} \left( \ev_Y{ \| \hat{\Theta}^\top Y - \Theta_*^\top Y \|_2^2 } \leq \trace{C} \gamma/(1-\gamma) + \xi \right) = 1.
\end{equation}
\end{prop}
\begin{proof}
By \Cref{prop:cond_out_of_sample_miss}, we have
\begin{multline}\label{eq:error_bound_rescaled}
\ev_Y{ \| \hat{\Theta}^\top Y - \Theta_*^\top Y \|_2^2 \mid \boldsymbol{Y}}  \leq \trace{C} \trace{(\boldsymbol{Z}^\top \boldsymbol{Z})^{-1}} \\ + 2 \sqrt{\trace{C^2}\trace{(\boldsymbol{Z}^\top \boldsymbol{Z})^{-2}} \ln(1/\delta_1) }   + 2 \|C\|_2 \|(\boldsymbol{Z}^\top \boldsymbol{Z})^{-1}\|_2 \ln(1/\delta_1)
\end{multline}
with probability at least $1-\delta_1$ for $\delta_1 \in (0,1)$. By \Cref{thm:conc_trace_prec} and the Bai--Yin theorem \cite[Theorem 5.11 on p.~106]{bai2010spectral}, we have, almost surely,
\begin{equation*}
\lim_{M \to \infty} \trace{(\boldsymbol{Z}^\top \boldsymbol{Z})^{-1}} = \gamma/(1-\gamma) \quad \text{and} \quad \lim_{M \to \infty} \| \boldsymbol{Z}^\top \boldsymbol{Z} \|_2 = 0.
\end{equation*}
Since $\trace{(\boldsymbol{Z}^\top \boldsymbol{Z})^{-2}} \leq \| \boldsymbol{Z}^\top \boldsymbol{Z} \|_2 \trace{(\boldsymbol{Z}^\top \boldsymbol{Z})^{-1}}$, this yields $\lim_{M \to \infty}\trace{(\boldsymbol{Z}^\top \boldsymbol{Z})^{-2}}$ = 0 almost surely. Since almost sure convergence implies convergence in probability, for $\delta_2 \in (0,1)$ and $\xi \geq 0$, there is $m \in \mathbb{N}$ such that for all $M \geq m$ we have 
\begin{multline}\label{eq:event_trace_bound}
\Big| \trace{C} \trace{(\boldsymbol{Z}^\top \boldsymbol{Z})^{-1}} + 2 \sqrt{\trace{C^2}\trace{(\boldsymbol{Z}^\top \boldsymbol{Z})^{-2}}\ln(1/\delta_1) }  \\ + 2 \|C\|_2 \|(\boldsymbol{Z}^\top \boldsymbol{Z})^{-1}\|_2 \ln(1/\delta_1) - \trace{C}
\gamma/(1-\gamma) \Big| \leq \xi   
\end{multline}
with probability at least $1-\delta_2$. By applying the union bound to the union of the complements of the events in \eqref{eq:error_bound_rescaled} and \eqref{eq:event_trace_bound}, we deduce that for all $m \geq M$
\begin{equation*}
\mathbb{P} \left( \ev_Y{ \| \hat{\Theta}^\top Y - \Theta_*^\top Y \|_2^2 } \leq \trace{C} \gamma/(1-\gamma) + \xi \right) \geq 1-\delta_1 - \delta_2.
\end{equation*}
Since $\xi \geq 0$ and $\delta_1, \delta_2 \in (0,1)$ are arbitrary, the proof is complete.
\end{proof}

\begin{prop}\label{prop:Gaussian_error_converges_in_p}
Consider the Gaussian model and let $\lim_{M \to \infty} M/n(M) = \gamma \in (0,1)$. Then $\ev_Y{ \| \hat{\Theta}^\top Y - \Theta_*^\top Y \|_2^2}$ converges to  $\trace{C_{EE}} \gamma/(1-\gamma)$ in probability as $M \to \infty$. 
\end{prop}
\begin{proof}
The proof is similar to that of \Cref{rem:asymptotic_err}, except that we need a two two-sided bound. We divide the proof into three steps. In the first step of the proof, we derive the two-sided tail bound that
\begin{multline}\label{eq:gaussian_model_tail_bound}
\left| \ev_Y{ \| \hat{\Theta}^\top Y - \Theta_*^\top Y \|_2^2} - \trace{C_{EE}}\trace{(\boldsymbol{Z}^\top \boldsymbol{Z})^{-1}} \right| \\ \leq   2 \sqrt{\trace{C_{EE}^2}\trace{(\boldsymbol{Z}^\top \boldsymbol{Z})^{-2}}\ln(2/\delta_1) }  + 2 \|C_{EE} \|_2 \|(\boldsymbol{Z}^\top \boldsymbol{Z})^{-1}\|_2 \ln(2/\delta_1)
\end{multline}
with probability $1-\delta_1$ for $\delta_1 \in (0,1)$. To derive this bound, we first recall from  the proof of \Cref{prop:cond_out_of_sample_miss} that $\ev_Y{ \| \hat{\Theta}^\top Y - \Theta_*^\top Y \|_2^2} = \| (\boldsymbol{Z}^\top \boldsymbol{Z})^{-1} \boldsymbol{Z}^\top \boldsymbol{E} \|_F^2$. Then, as in the proof of \Cref{lem:tail_inequality_quadratic_max_forms}, we observe that $\| (\boldsymbol{Z}^\top \boldsymbol{Z})^{-1} \boldsymbol{Z}^\top \boldsymbol{E} \|_F^2 = \| S \vec{G} \|_2^2$, where $S \in \mathbb{R}^{M N \times n N} $ is the block diagonal matrix $S \coloneqq \text{diag}\left( d_1^{1/2} (\boldsymbol{Z}^\top \boldsymbol{Z})^{-1} \boldsymbol{Z}^\top, \dots, d_N^{1/2} (\boldsymbol{Z}^\top \boldsymbol{Z})^{-1} \boldsymbol{Z}^\top \right)$, $\vec{G}$ is a random vector in $\mathbb{R}^{nN}$ with independent standard normal random variables as entries, and $d_1, \dots, d_N$ are the eigenvalues of $C_{EE}$. Then by \cite[Lemma 1]{laurent2000adaptive} we have
\begin{equation}\label{eq:tail_bound_standard_gauss}
\left| \| S \vec{G} \|_2^2 - \trace{S^\top S} \right| \leq   2 \sqrt{\trace{(S^\top S)^2}\ln(2/\delta_1)}   + 2 \|S^\top S\|_2 \ln(2/\delta_1)
\end{equation}
with probability at least $1-\delta_1$. To complete the proof of  \eqref{eq:gaussian_model_tail_bound}, we argue as in the proof of \Cref{lem:tail_inequality_quadratic_max_forms} and apply  \Cref{lem:trace_and_trace_square_inequalities} to deduce that \eqref{eq:tail_bound_standard_gauss} is equivalent to \eqref{eq:gaussian_model_tail_bound}. For the second step of the proof, we first recall from the proof of \Cref{rem:asymptotic_err} that, almost surely,  $\trace{(\boldsymbol{Z}^\top \boldsymbol{Z})^{-1}} \to \gamma/(1-\gamma)$ and 
$\| \boldsymbol{Z}^\top \boldsymbol{Z} \|_2 \to 0$ as $M \to \infty$. Since almost sure convergence implies convergence in probability, for arbitrary $\xi_1, \xi_2 \geq 0$ and $\delta_2 \in (0,1)$ we can find $m \in \mathbb{N}$ such that for all $M \geq m$ we have
\begin{equation}\label{eq:bound2_gauss_conc}
\begin{split}
\left| \trace{C_{EE}}\trace{(\boldsymbol{Z}^\top \boldsymbol{Z})^{-1}}  - \trace{C_{EE}} \gamma/(1-\gamma) \right| &\leq \xi_1, \\
 2 \sqrt{\trace{C_{EE}^2}\trace{(\boldsymbol{Z}^\top \boldsymbol{Z})^{-2}}\ln(2/\delta_1) }  + 2 \|C_{EE} \|_2 \|(\boldsymbol{Z}^\top \boldsymbol{Z})^{-1}\|_2 \ln(2/\delta_1) &\leq \xi_2.
\end{split}
\end{equation}
with probability at least $1-\delta_2$. By applying the union bound to the union of the complements of the events in \eqref{eq:gaussian_model_tail_bound} and \eqref{eq:bound2_gauss_conc}, we deduce that for $M \geq m$
\begin{equation*}
\mathbb{P} \left( \left| \ev_Y{ \| \hat{\Theta}^\top Y - \Theta_*^\top Y \|_2^2}  - \trace{C_{EE}} \gamma/(1-\gamma) \right| \leq \xi_1 + \xi_2 \right) \geq 1- \delta_1 - \delta_2.
\end{equation*}
Since $\xi_1, \xi_2 \geq 0$ and $\delta_1, \delta_2 \in (0,1)$ are arbitrary, the proof is complete. \end{proof}
The following theorem, which is the main result of this work,  provides a tail bound for the mean squared error of the least squares estimator.
\begin{thm}[Tail bounds for the approximation error]\label{thm:out_of_sample_error}
Assume that 
\begin{enumerate}[label=\roman*)]
\item there exists a symmetric positive-definite matrix $C \in \mathbb{R}^{N \times N}$ such that, almost surely,
\begin{equation*}
\ev{ \exp \left( \alpha^\top C^{-1/2} (E\mid Y - \ev{E \mid Y}) \right)} \leq \exp\left(\| \alpha \|_2^2 / 2 \right) \quad \text{for all } \alpha \in \mathbb{R}^N,
\end{equation*}\label{item:cond_MSE_MMSE_bound}
\item there is $\mu=(\mu_1, \dots, \mu_N)  \in \mathbb{R}^N$ such that $|(E \mid Y)_i | \leq \mu_i$ almost surely for $1 \leq i \leq N$, \label{item:approx_err}  
\item $ C_{YY}^{-1/2} Y $ is a $\rho$-sub-Gaussian random vector.
\end{enumerate}
Then, for $t \geq 0$ such that $\sqrt{n}- c \rho^2 ( \sqrt{M} + \sqrt{t}) > 0$ and $s \geq 0$, we have 
\begin{equation*}
\ev_Y{ \| \hat{\Theta}^\top Y - \Theta_*^\top Y \|_2^2} \leq  \varepsilon_{\text{err}} + \varepsilon_{\text{bias}}  ,
\end{equation*}
with probability at least 
$
1 - 3 \exp(-t) - N\exp(-s)$, where 
\begin{align*}
\varepsilon_{\text{err}} &\coloneqq  \frac{\trace{C} M + 2 \sqrt{M\trace{C^2} t}  + 2 \|C\|_2 t}{(\sqrt{n} - c \rho^2 (\sqrt{M} +  \sqrt{t}))^2}, \\
\varepsilon_{\text{bias}} &\coloneqq  \frac{ (3/2) \rho^2\| \mu\|_2^2 \, n \left(M + 2 \sqrt{M s} + 2 s \right)}{(\sqrt{n} - c \rho^2 (\sqrt{M}+\sqrt{t}))^4} \left(1+\frac{2 \sqrt{t} \|C\|_2}{\sqrt{M\trace{C^2}}} \left(\frac{\sqrt{n} + c \rho^2(\sqrt{M} + \sqrt{t})} {\sqrt{n}-c\rho^2 ( \sqrt{M}+\sqrt{t})} \right)^2 \right),
\end{align*}
and $c>0$ is as in \Cref{thm:sub_gaussian-rows}.

\end{thm}
\begin{proof}
By \Cref{prop:cond_out_of_sample_miss}, we only need to establish suitable tail bounds for the quantities $\varepsilon_{\text{err}}^{\boldsymbol{Y}}$ and $\varepsilon_{\text{bias}}^{\boldsymbol{Y}}$ defined in \eqref{eq:eps_err} and \eqref{eq:eps_bias}. For this purpose, we first recall that  
\begin{equation}\label{eq:trace_bounds}
\begin{split}
\| (\boldsymbol{Z}^\top \boldsymbol{Z})^{-1} \|_2 = \lambda_{\min}(\boldsymbol{Z}^\top \boldsymbol{Z})^{-1}, \quad 
\trace{(\boldsymbol{Z}^\top \boldsymbol{Z})^{-1}} \leq M \lambda_{\min}(\boldsymbol{Z}^\top \boldsymbol{Z})^{-1}, \\
\trace{(\boldsymbol{Z}^\top \boldsymbol{Z})^{-2}} \leq M \lambda_{\min}(\boldsymbol{Z}^\top \boldsymbol{Z})^{-2}, \text{ and } \quad  
\trace{(\boldsymbol{Z}^\top \boldsymbol{Z})^{-2}} \geq M \lambda_{\max}(\boldsymbol{Z}^\top \boldsymbol{Z})^{-2},
\end{split}
\end{equation}
where $\lambda_{\min}(\boldsymbol{Z}^\top \boldsymbol{Z})$ and $\lambda_{\max}(\boldsymbol{Z}^\top \boldsymbol{Z})$ are the minimal and maximal eigenvalues of $\boldsymbol{Z}^\top \boldsymbol{Z}$. Moreover, we have
\begin{equation}\label{eq:product_f_bound}
\| (\boldsymbol{Z}^\top \boldsymbol{Z})^{-1} \boldsymbol{Z}^\top \ev{\boldsymbol{E} \mid \boldsymbol{Y}} \|_F^2 \leq  \| (\boldsymbol{Z}^\top \boldsymbol{Z})^{-1} \|_2^2  \| \boldsymbol{Z}^\top \ev{\boldsymbol{E} \mid \boldsymbol{Y}} \|_F^2.
\end{equation}
By definition of $\varepsilon_{\text{err}}^{\boldsymbol{Y}}$ and  $\varepsilon_{\text{bias}}^{\boldsymbol{Y}}$, it follows that to prove the claim, we only need to derive appropriate bounds for $\lambda_{\min}(\boldsymbol{Z}^\top \boldsymbol{Z})^{-1}$, $\lambda_{\max}(\boldsymbol{Z}^\top \boldsymbol{Z})$, and $\| \boldsymbol{Z}^\top \ev{\boldsymbol{E} \mid \boldsymbol{Y}} \|_F^2$:
\paragraph{Bounds for $\lambda_{\min}(\boldsymbol{Z}^\top \boldsymbol{Z})$ and $\lambda_{\max}(\boldsymbol{Z}^\top \boldsymbol{Z})$:} By \Cref{thm:sub_gaussian-rows}, we have
\begin{equation}\label{eq:singular_values_in_proof}
(\sqrt{n}- c \rho^2 (\sqrt{M} + \sqrt{t}))^2  \leq \lambda_{\min}(\boldsymbol{Z}^\top \boldsymbol{Z}) \leq \lambda_{\max}(\boldsymbol{Z}^\top \boldsymbol{Z}) \leq (\sqrt{n} + c \rho^2(\sqrt{M} + \sqrt{t}))^2
\end{equation}
with probability at least $1-2 e^{-t}$. 

\paragraph{Bounds for $\| \boldsymbol{Z}^\top \mathbb{E}[\boldsymbol{E} \mid \boldsymbol{Y}] \|_F^2$:}
By applying \Cref{cor:weighted_sum_sub-gaussians} to the $i$th column $(\boldsymbol{Z}^\top \ev{\boldsymbol{E} \mid \boldsymbol{Y}})^i$ of $\boldsymbol{Z}^\top \ev{\boldsymbol{E} \mid \boldsymbol{Y}}$ for $1 \leq i \leq N$, we obtain 
\begin{equation}\label{eq:approx_error_columns}
\| (\boldsymbol{Z}^\top \ev{\boldsymbol{E} \mid \boldsymbol{Y}})^i \|_2^2  \leq (3/2) \rho^2 \mu_i^2 (n M + 2 n\sqrt{M s} + 2n s) 
\end{equation}
with probability at least $1-e^{-t}$. By applying the union bound to the union of the complements of the events in \eqref{eq:approx_error_columns} over all $1 \leq i \leq N$, we obtain  
\begin{equation}\label{eq:approx_error}
\| \boldsymbol{Z}^\top \ev{\boldsymbol{E} \mid \boldsymbol{Y}} \|_F^2  \leq (3/2) \rho^2 \| \mu \|_2^2  (n M + 2n \sqrt{M s} + 2n s)
\end{equation}
with probability at least $1- N e^{-s}$. 
\paragraph{Union bound:} 
By applying the union bound to the union of the complements of the events in \eqref{eq:conditioned_bound}, \eqref{eq:singular_values_in_proof}, and \eqref{eq:approx_error}, we obtain that the intersection of the events in \eqref{eq:conditioned_bound}, \eqref{eq:singular_values_in_proof}, and \eqref{eq:approx_error} occurs with probability at least $1- 3 e^{-t} - N e^{-s}$. Using \crefrange{eq:trace_bounds}{eq:product_f_bound}, it follows that with at least the same probability,   
\begin{align*}
\varepsilon_{\text{err}}^{\boldsymbol{Y}} &\leq  \frac{\trace{C} M + 2 \sqrt{M\trace{C^2} t}  + 2 \|C\|_2 t}{(\sqrt{n} - c \rho^2 (\sqrt{M} +  \sqrt{t}))^2}, \\
\varepsilon_{\text{bias}}^{\boldsymbol{Y}} &\leq  \frac{ (3/2) \rho^2\| \mu\|_2^2 \, n \left(M + 2 \sqrt{M s} + 2 s \right)}{(\sqrt{n} - c \rho^2 (\sqrt{M}+\sqrt{t}))^4} \left(1+\frac{2 \sqrt{t} \|C\|_2}{\sqrt{M\trace{C^2}}} \left(\frac{\sqrt{n} + c \rho^2(\sqrt{M} + \sqrt{t})} {\sqrt{n}-c\rho^2 ( \sqrt{M}+\sqrt{t})} \right)^2 \right).
\end{align*}
The proof is complete. 
\end{proof}
In view of \Cref{prop:estimation_error}, the hypotheses in \Cref{thm:out_of_sample_error} can be interpreted as follows: assumption \ref{item:cond_MSE_MMSE_bound} is equivalent to the assumption that the conditional mean squared error of the MMSE estimator is sub-Gaussian with companion matrix bounded by $C \in \mathbb{R}^{N \times N}$ in the sense of the Loewner ordering. Assumption in \ref{item:approx_err} is equivalent to the assumption that, for $1 \leq i \leq N$, the absolute value of the $i$th component of the difference between the LMMSE and the MMSE estimates is bounded by $\mu_i \geq 0$.

\subsection{Consequences of \Cref{thm:out_of_sample_error}}
We state a few consequences of \Cref{thm:out_of_sample_error}.
\begin{prop}[Answer to \Cref{quest:conf_err} for linear model]\label{rem:answer_question_2}
Suppose that the hypotheses of \Cref{thm:out_of_sample_error} hold. Moreover, assume that the LMMSE and the MMSE estimators coincide (which is equivalent to the assumption that $\ev{E \mid Y } = 0$ almost surely). Then
\begin{equation*}
n = \left(\sqrt{\frac{ \trace{C} M + 2 \sqrt{M \trace{C^2} \ln{(3/\delta)} }    + 2 \| C \|_2 \ln{(3/\delta)}}{ \trace{C_{EE}} \varepsilon}} + c \rho^2\left(\sqrt{M} + \sqrt{\ln{(3/\delta)}}\right) \right)^2 
\end{equation*}
samples are sufficient in \Cref{quest:conf_err}.
\end{prop}
\begin{proof}
This follows by direct calculations from the error bound in \Cref{thm:out_of_sample_error}.
\end{proof}

\begin{prop}\label{rem:answer_question_2_Gauss}
For the Gaussian model, 
\begin{equation}\label{eq:estimate_Gaussian_2}
n= \left(\sqrt{ \frac{\trace{C_{EE}} M  + 2 \sqrt{M \trace{C_{EE}^2} \ln{(3/\delta)  }}     + 2 \| C_{EE} \|_2 \ln{(3/\delta)}}{\trace{C_{EE}} \varepsilon}} + \sqrt{M} + \sqrt{\ln{(3/\delta)}} \right)^2 
\end{equation}
samples are sufficient in \Cref{quest:conf_err}. Moreover, the estimate in \eqref{eq:sample_bound} is an upper bound for the estimate in \eqref{eq:estimate_Gaussian_2}.
 \end{prop}
\begin{proof}
The first claim is a corollary of \Cref{rem:answer_question_2}, since for the Gaussian model the hypotheses of \Cref{rem:answer_question_2} hold with $C=C_{EE}$, $\rho = 1$, and $c=1$; see \Cref{prop:posterior_gaussian_model} and \Cref{thm:sub_gaussian-rows}. The second claim is by the relation $\trace{C_{EE}^2} \leq \| C_{EE}\|_2 \trace{C_{EE}}$.
\end{proof}
\Cref{thm:out_of_sample_error} can be used to derive an upper bound for the number of samples in \Cref{quest:conf_err}  also when the LMMSE and the MMSE estimators do not coincide. However, since the arguments required for this derivation do not provide further insight, we focus instead on the asymptotic behavior of the error.

\begin{prop}[Asymptotic error]\label{rem:asymptotic_error}
Suppose that,  for every $M \in \mathbb{N}$,  the hypotheses of \Cref{thm:out_of_sample_error} hold with $(C,\mu,\rho)$ independent of $M$. Let $n=n(M)$ be such that $\lim_{M \to \infty} M/n(M) = \gamma$ for some $0 < \gamma < 1/(c^2 \rho^4)$. Then
\begin{equation*}
\lim_{M \to \infty} \ev_Y{ \| \hat{\Theta}^\top Y - \Theta_*^\top Y \|_2^2} \leq  \varepsilon_{\text{err}}^\gamma + \varepsilon_{\text{bias}}^\gamma \quad \text{almost surely},
\end{equation*}
where
$
\varepsilon_{\text{err}}^\gamma \coloneqq  \trace{C} \gamma /(1 - c \rho^2 \sqrt{\gamma})^2 
$
and
$
\varepsilon_{\text{bias}}^\gamma \coloneqq (3/2) \rho^2\| \mu\|_2^2  \gamma /(1- c \rho^2 \sqrt{\gamma})^4.
$
\end{prop}
\begin{proof}
Let $v,w \geq 0$ be arbitrary and define $\gamma(M) \coloneqq M/n(M)$. Since $\lim_{M \to \infty} \gamma(M) = \gamma$ and $\gamma <1/(c \rho^2)$, there is $m_1 \in \mathbb{N}$ such that for all $M\geq m_1$ we have $1- c \rho^2( \sqrt{\gamma(M)} + \sqrt{v \gamma(M)/\sqrt{n(M)}}) > 0$. By applying \Cref{thm:out_of_sample_error} for $t = v\sqrt{n(M)}$ and $w = s \sqrt{n(M)}$, we obtain that
\begin{equation*}
\ev_Y{ \| \hat{\Theta}^\top Y - \Theta_*^\top Y \|_2^2} \leq  \trace{C} a(M) + (3/2) \rho^2\| \mu\|_2^2 b(M) (1+c(M)) ,
\end{equation*}
with probability at least 
$
1 - 3 \exp(-v \sqrt{n(M)}) - N\exp(-s \sqrt{n(M)})$, where 
\begin{align*}
a(M) &\coloneqq  \frac{\gamma(M) + 2 \sqrt{v\,  \trace{C^2} \gamma(M)/\sqrt{n(M)} } + 2  v \| C \|_2/ \sqrt{n(M)} }{(1 - c \rho^2 (\sqrt{\gamma(M)} +  \sqrt{v/\sqrt{n(M)}}))^2}, \\
b(M) &\coloneqq  \frac{\gamma(M) + 2 \sqrt{w \gamma(M)/\sqrt{n(M)}} + 2 w/\sqrt{n(M)} }{(1 - c \rho^2 (\sqrt{\gamma(M)}+\sqrt{v/\sqrt{n(M)}}))^4}, \\ 
c(M) &\coloneqq \frac{2 \sqrt{v/\sqrt{n(M)}}  \|C\|_2}{\sqrt{\gamma(M)}\sqrt{\trace{C^2}}}  \left(\frac{1 + c \rho^2(\sqrt{\gamma(M)} + \sqrt{v/\sqrt{n(M)}})} {1-c\rho^2 ( \sqrt{\gamma(M)}+\sqrt{v/\sqrt{n(M)}})} \right)^2.
\end{align*}
Since $\lim_{M \to \infty} a(M) = \gamma/(1 - c \rho^2 \sqrt{\gamma})^2$, $\lim_{M \to \infty} b(M) = \gamma/(1- c \rho^2 \sqrt{\gamma})^4$, and $\lim_{M \to \infty} c(M) = 0$, for every $\varepsilon > 0$, we can find $m_2 \geq m_1$ such that for all $M \geq m_2$
\begin{equation*}
\mathbb{P} \left( \ev_Y{ \| \hat{\Theta}^\top Y - \Theta_*^\top Y \|_2^2} \geq  \varepsilon_{\text{err}}^\gamma + \varepsilon_{\text{bias}}^\gamma + \varepsilon \right) \leq  3 \exp(-v \sqrt{n(M)}) - N\exp(-w \sqrt{n(M)}.
\end{equation*}
Since the ratio test shows that the series 
$
 \sum_{M=m_2}^\infty  3 \exp\left(-v \sqrt{n(M)}\right) + N\exp \left(-s \sqrt{n(M)} \right)  
$
converges, the claim follows by the Borel--Cantelli lemma \cite[p.~309]{shiryaev2016}.
\end{proof}
Note that in \Cref{rem:asymptotic_error} the asymptotic error bound behaves as $\mathcal{O}(\gamma)$ for $\gamma \to 0$.

\section{Numerical experiments}\label{sec:num-exp}
We numerically investigate the behavior of the mean squared error of the least squares estimator by considering two models: a Gaussian model with a random forward operator, and a denoising model for low resolution images.
\subsection{Gaussian model}\label{subsec:data}
\paragraph{The data.}
To create the covariance matrices $C_{XX} \in \mathbb{R}^{N \times N}$ and $C_{ZZ} \in \mathbb{R}^{M \times M}$ for the Gaussian model we proceed as follows: first, we create  orthogonal matrices $P \in \mathbb{R}^{N \times N}$ and $Q \in \mathbb{R}^{M \times M}$ randomly using the procedure from \cite[Section 4.6]{edelman_rao_2005}. Then, we create  diagonal matrices $\Lambda_X \in \mathbb{R}^{N \times N}$ and $\Lambda_Z \in \mathbb{R}^{M \times M}$ by letting their diagonal entries be independent samples from a random variable that is uniformly distributed in $[0,1]$. Finally, we set $C_{XX} = P \Lambda_X P^\top$ and $C_{ZZ} = Q \Lambda_Z Q^\top$. The entries of the matrix $A \in \mathbb{R}^{M \times N}$ are independent samples from a standard normal random variable. 

\paragraph{Experimental setup.}
For each $N,M \in \{ 16, 128\}$, we create $A$, $C_{XX}$, and $C_{ZZ}$ as described in the previous paragraph. We calculate the LMMSE estimator $\Theta_*$ and its mean squared error $\trace{C_{EE}}$ based on the formulas in \Cref{prop:characterization_of_LMMSE}. For each $\varepsilon\in  \{1/16, 1/4, 1/2, 1 \}$, we choose the number of samples $n$ according to the formula in \eqref{eq:formula_n_gaussian}, and repeat the following procedure $300$ times:
 \begin{enumerate}[label=\arabic*.)]
\item Create $n$ independent samples $\{(X_i,Z_i)\}_{i=1}^n$ of $(X,Z)$ and let $Y_i \coloneqq A X_i + Z_i$.

\item Compute the least squares estimator $\hat{\Theta}$ based on the formula in \eqref{eq:least_squares estimator}.
\item Calculate the mean squared error of the least squares estimator by using the decomposition
in \eqref{MSE_decompostion_linear}.

\end{enumerate}
For each $\varepsilon\in  \{1/16, 1/4, 1/2, 1 \}$, we obtain $300$ realizations of the mean squared error of the least squares estimator. We determine the empirical tail distribution of these realizations by determining, for given $\tau \geq 1$, the fraction of realizations that is greater than $\tau \trace{C_{EE}}$. 
\paragraph{Results.}
In \Cref{fig:empirical_tail_mse_Gauss}, we plot the empirical tail distribution of the mean squared error of the least squares estimators alongside the theoretical prediction for the expected mean squared error. We find that, for fixed $\varepsilon>0$ and $n$ chosen as in \eqref{eq:formula_n_gaussian}, the mean squared error is more concentrated (relative to the size of $\trace{C_{EE}}$) around the expected value for larger values of $M$ and $N$. This is consistent with the sample bound from \Cref{rem:answer_question_2_Gauss}.

\begin{figure}[tbhp]
\centering
\subfloat[$M = 16$ and $N = 16$]{
\includegraphics[trim=20 50 90 110,clip, width=0.5\columnwidth]{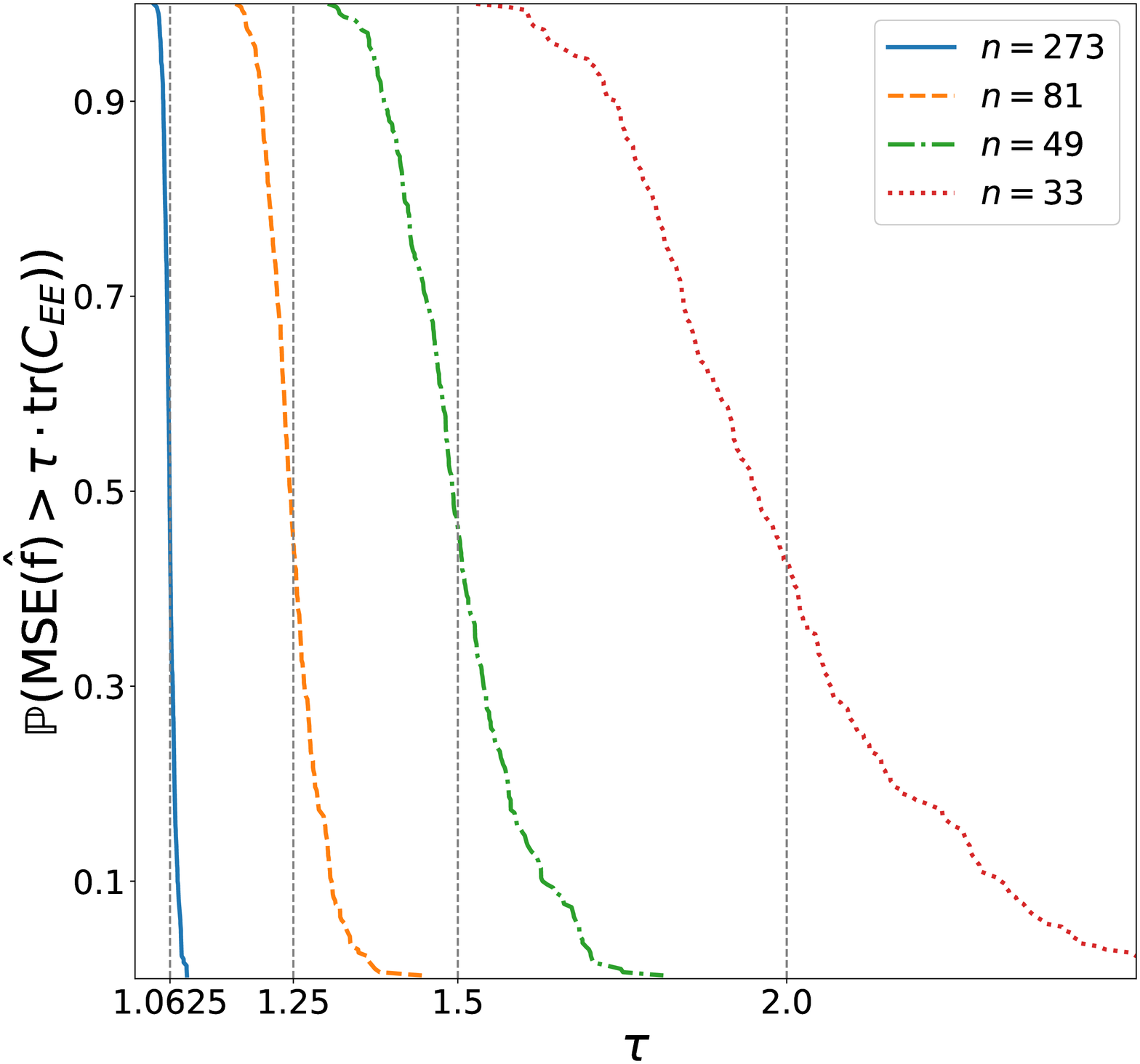}}
\subfloat[$M = 16$ and $N = 128$]{
\includegraphics[trim=20 50 90 110 , clip, width=0.5\columnwidth]{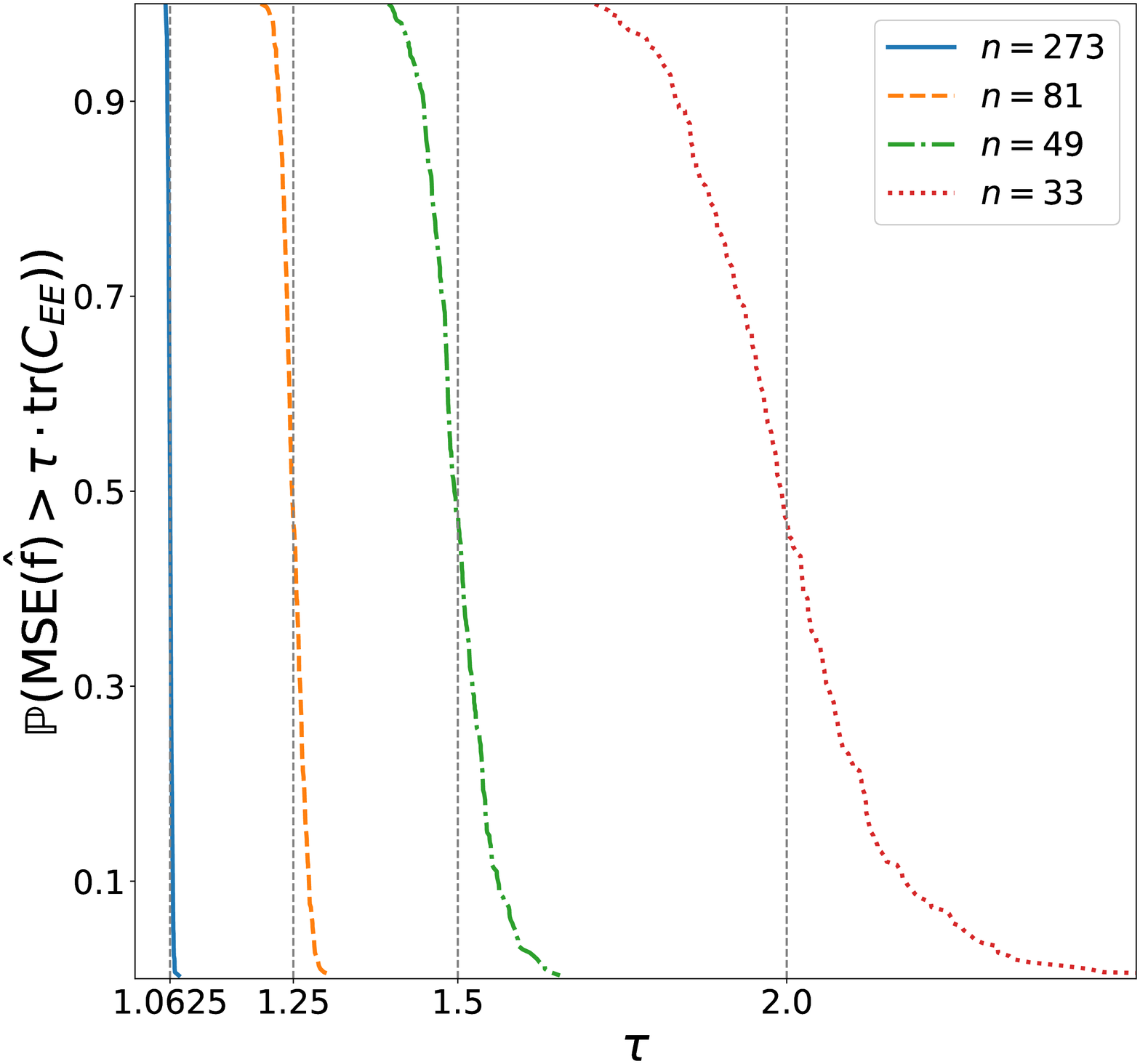}}

\subfloat[$M = 128$ and $N = 16$]{
\includegraphics[trim=20 50 90 110, clip, width=0.5\columnwidth]{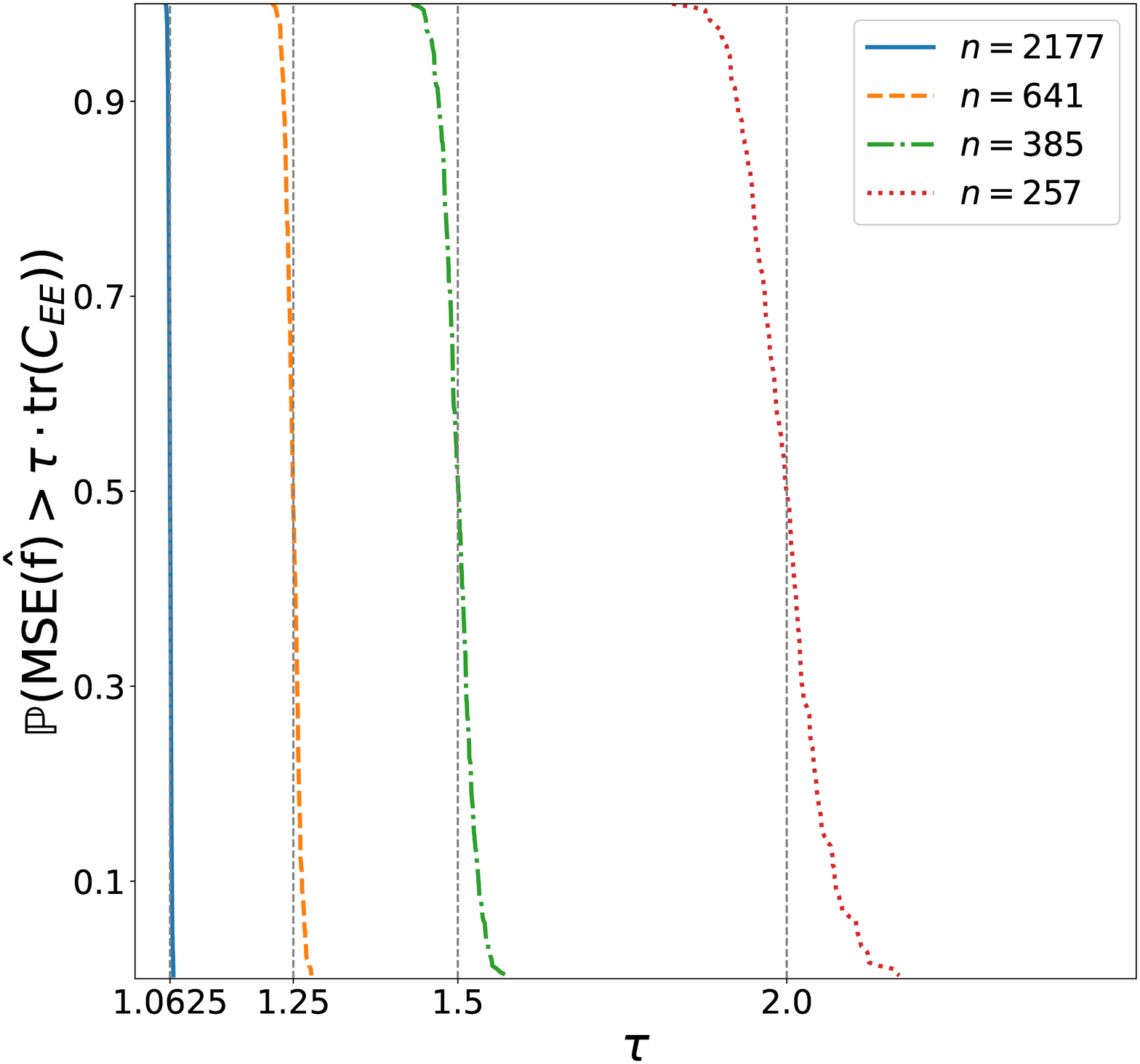}}
\subfloat[$M = 128$ and $N = 128$]{
\includegraphics[trim=20 50 90 110, clip, width=0.5\columnwidth]{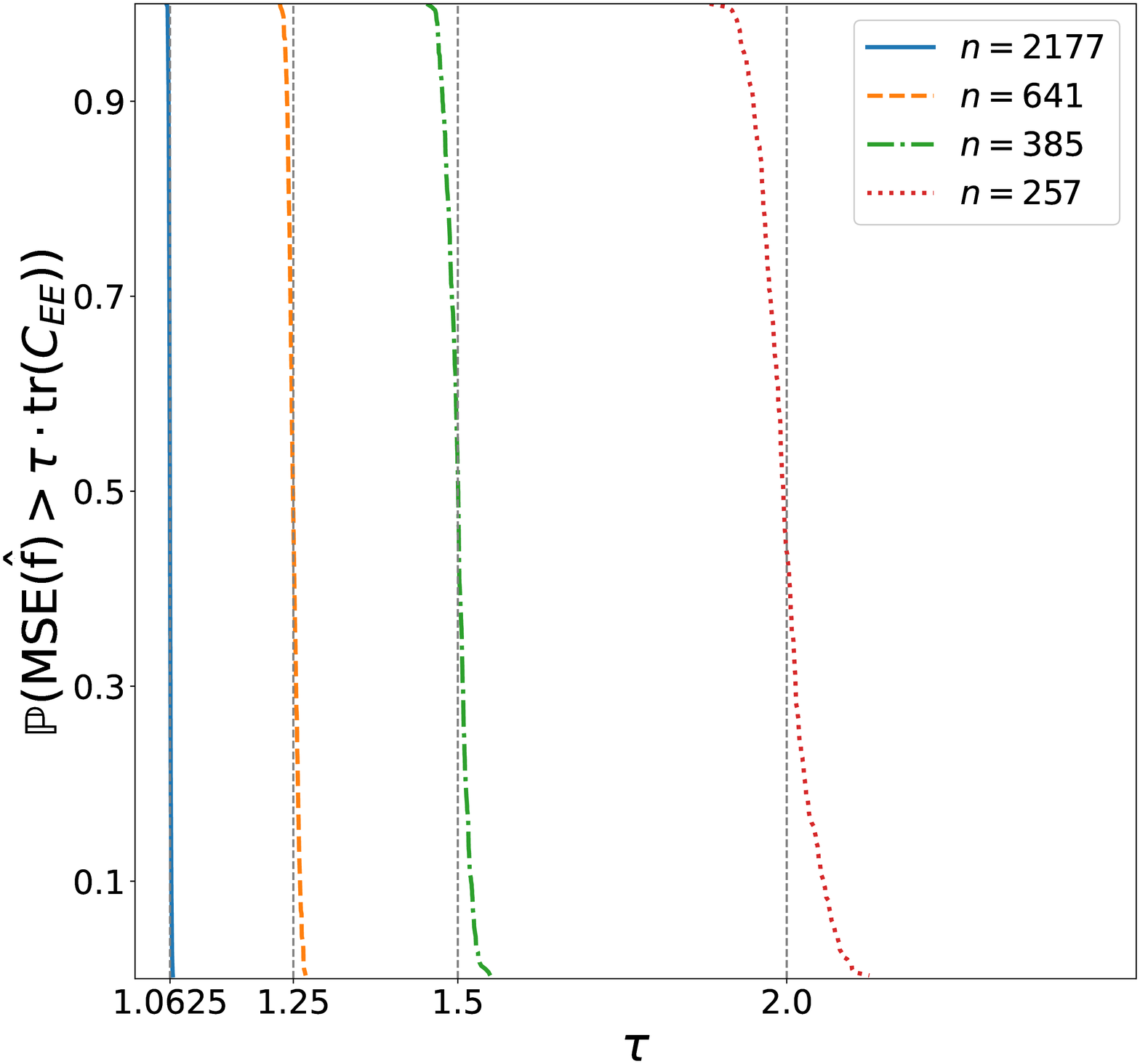}}

\caption{Empirical tail distribution of the mean squared error of the least squares estimator for the Gaussian model for $M,N \in \{16,128\}$. The dashed vertical lines mark the predictions for the expected mean squared error.}\label{fig:empirical_tail_mse_Gauss}
\end{figure}

\subsection{Denoising model}
We perform experiments for learning the LMMSE estimator for a denoising problem. The purpose of the experiments is to demonstrate that the provided insight into how many samples are needed to obtain reliable approximations to the LMMSE estimator is useful beyond the Gaussian case. It is not the purpose of these experiments to argue that the LMMSE estimator is a good choice for denoising images; if anything, our results indicate that there is a fundamental limit to the quality of any denoising method which leads to a linear mapping from the data to an estimate.

\paragraph{The data.}
We use the Fashion-MNIST dataset\footnote{Fashion-MNIST is available at \url{https://github.com/zalandoresearch/fashion-mnist}.
} \cite{xiao2017}. Fashion-MNIST is a dataset of e-commerce company Zalando's article images, consisting of $70000$ samples. The samples are divided into 60000 training samples and 10000 test samples. Each sample is a $28 \times 28$ grayscale image. Each grayscale image displays a piece of clothing from one of 10 classes. We first rescale every image such that the pixel values are in $[0,1]$. We then subtract the mean of all images from the samples so that we can assume that the samples have mean zero. Treating the images $X $ as a discrete random variable with values in $\mathbb{R}^{28^2}$,  we compute $C_{XX}$ directly from the rescaled training images, by letting
$
C_{XX} =(1/60000) \sum_{i=1}^{60000} X_i X_i^\top. 
$
To create noisy images, for $\sigma = 0.1$, we add independent samples of a random variable that is uniformly distributed in $[-\sigma \sqrt{3}, \sigma \sqrt{3}]$ to each pixel of each image. This choice ensures that $C_{ZZ} = \sigma^2 \id$. The mean squared error of the LMMSE estimator on the training set is given by 
$
\trace{C_{EE}}= \sum_{i=1}^N \xi_i \sigma^2 / (\xi_i + \sigma^2),
$
 where $\xi_i$ are the eigenvalues of $C_{XX}$.

\paragraph{Experimental setup.}
We calculate the LMMSE estimator $\Theta_*$ and its mean squared error $\trace{C_{EE}}$ based on the formulas in \Cref{prop:characterization_of_LMMSE}. For each $\varepsilon\in  \{1/16, 1/4, 1/2, 1 \}$, we choose $n$ according to \eqref{eq:formula_n_gaussian} and repeat the following procedure 300 times: 
\begin{enumerate}[label=\arabic*.)]
\item Choose $n$ samples $\{X_i\}_{i=1}^n$ (without repetitions) from the set of training images, create $n$ independent samples $\{Z_i\}_{i=1}^n$ of $Z$, and let $Y_i = X_i + Z_i$.\label{eq:step1_gaussian_denois}
\item Compute the least squares estimator $\hat{\Theta}$ based on the formula in \eqref{eq:least_squares estimator}.
\item Calculate the mean squared error of the least squares estimator by using the decomposition
in \eqref{MSE_decompostion_linear}.

\item To compute the test error, use the $10000$ test image samples $\{X_{\text{test}, i}\}_{i=1}^{10000}$, which were neither used for computing the least squares estimator nor for obtaining $\trace{C_{EE}}$. Create the noisy data $\{Y_{\text{test}, i}\}_{i=1}^{10000}$ for the test images as in \ref{eq:step1_gaussian_denois}. Compute the test error of the least squares estimator by evaluating
\begin{equation}\label{eq:test_error_formula}
\text{MSE}_{\text{test}}(\hat{\Theta}) = \frac{1}{10000} \sum_{i=1}^{10000} \| \hat{\Theta}^\top Y_{\text{test}, i} - X_{\text{test}, i} \|_2^2.
\end{equation}

\end{enumerate}
For each $\varepsilon\in  \{1/16, 1/4, 1/2, 1 \}$, we obtain $300$ realizations of the mean squared error. We calculate the empirical tail distribution as for the Gaussian model. 
\paragraph{Results.}
In \Cref{fig:mse_denoising_least_squares_28square}, we plot the empirical tail distribution of the mean squared error of the least squares estimator. In \Cref{fig:test_error_denoising_least_squares_28square}, we plot the empirical tail distribution of the test error of the least squares estimator. In \Cref{fig:samples_from_fashionMNIST}, we display 6 sample images from the test set. We make the following observations: 
\begin{enumerate}[label=\arabic*.)]
\item In all tested cases, the test error and the mean squared error were similar, even though the test samples were not used for computing the mean squared error. 

\item For $n$ chosen as in \eqref{eq:formula_n_gaussian}, the mean squared error and the test error of the least squares estimators behaved almost deterministically.

\item  For $n$ chosen as in \eqref{eq:formula_n_gaussian}, the mean squared error of the least squares estimator approaches the asymptotic prediction from \eqref{eq:err_conv_prob} as $\varepsilon>0$ tends to zero.   
\end{enumerate}

\begin{figure}[tbhp]
\centering
\subfloat[Tail distribution of the mean squared error. \label{fig:mse_denoising_least_squares_28square}]{
\includegraphics[width=0.5\columnwidth]{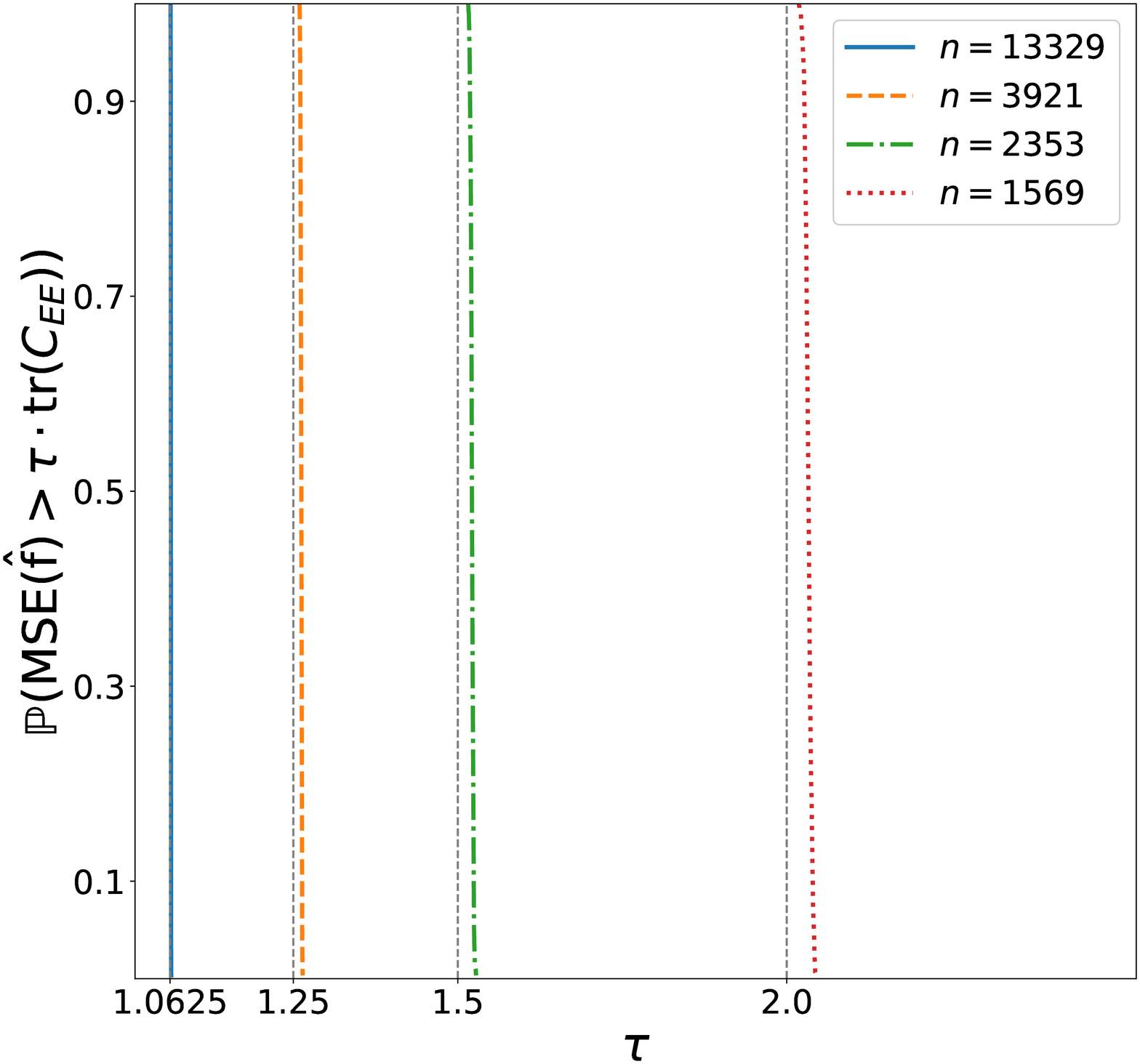}}\subfloat[Tail distribution of the test error. \label{fig:test_error_denoising_least_squares_28square}]{
\includegraphics[width=0.5\columnwidth]{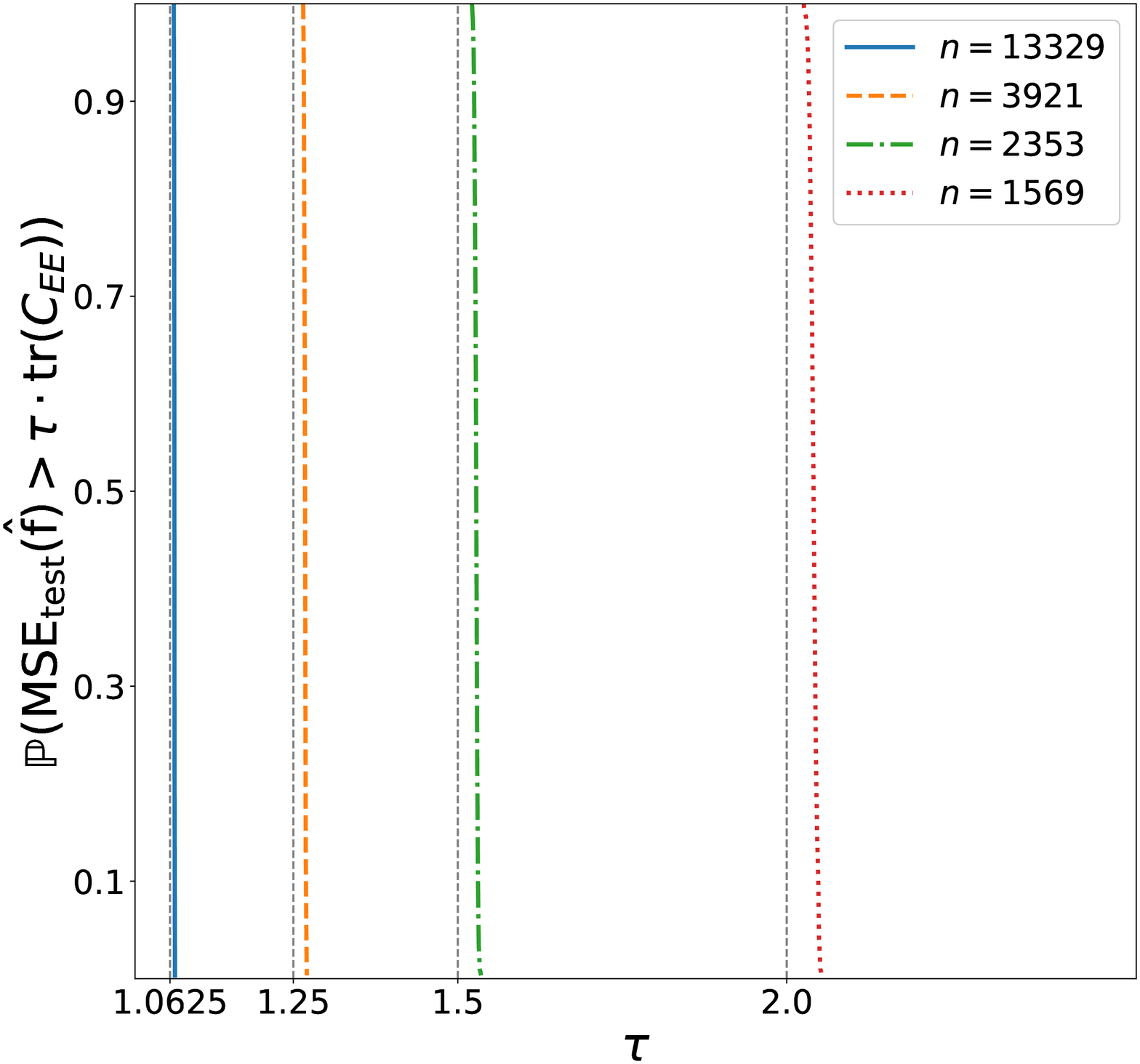}}

\caption{Empirical tail distribution of the mean squared error and the test error of the least squares estimator for the denoising model. The dashed vertical lines mark the predictions for the asymptotic mean squared error from \eqref{eq:err_conv_prob}.}

\end{figure}

\begin{figure}[tbhp]
\centering
\includegraphics[width=0.16\columnwidth]{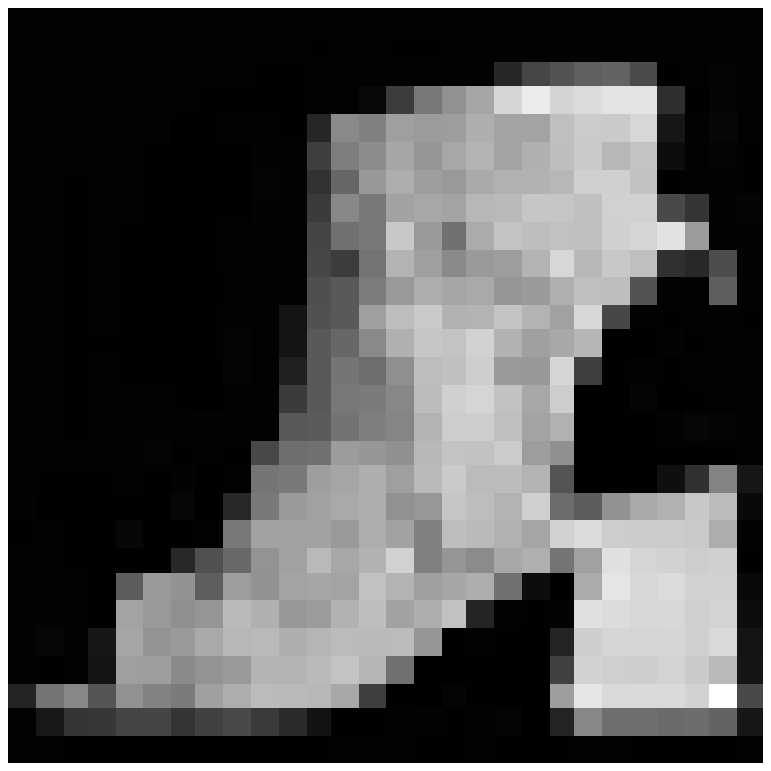}\includegraphics[width=0.16\columnwidth]{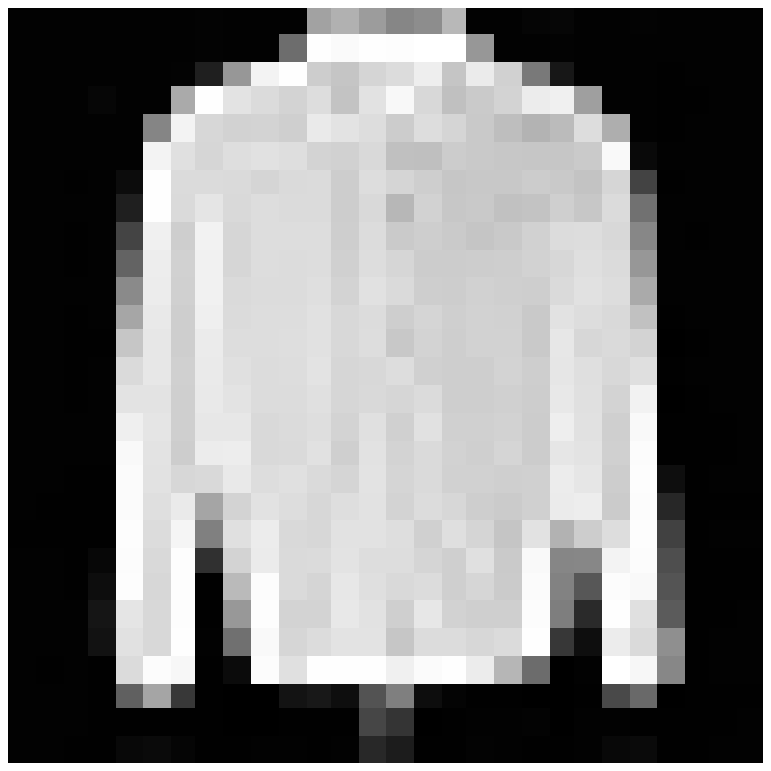}\includegraphics[width=0.16\columnwidth]{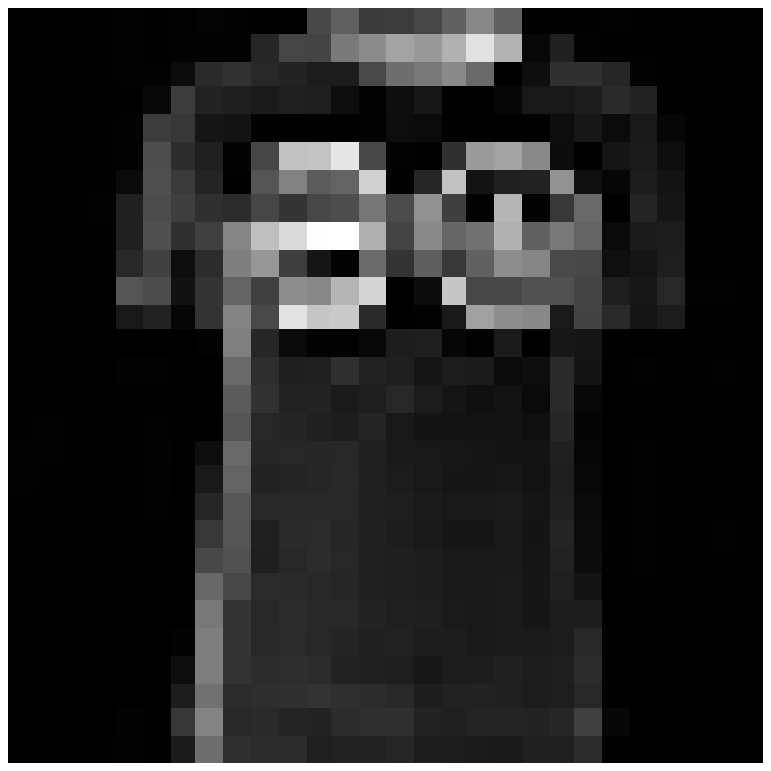}\includegraphics[width=0.16\columnwidth]{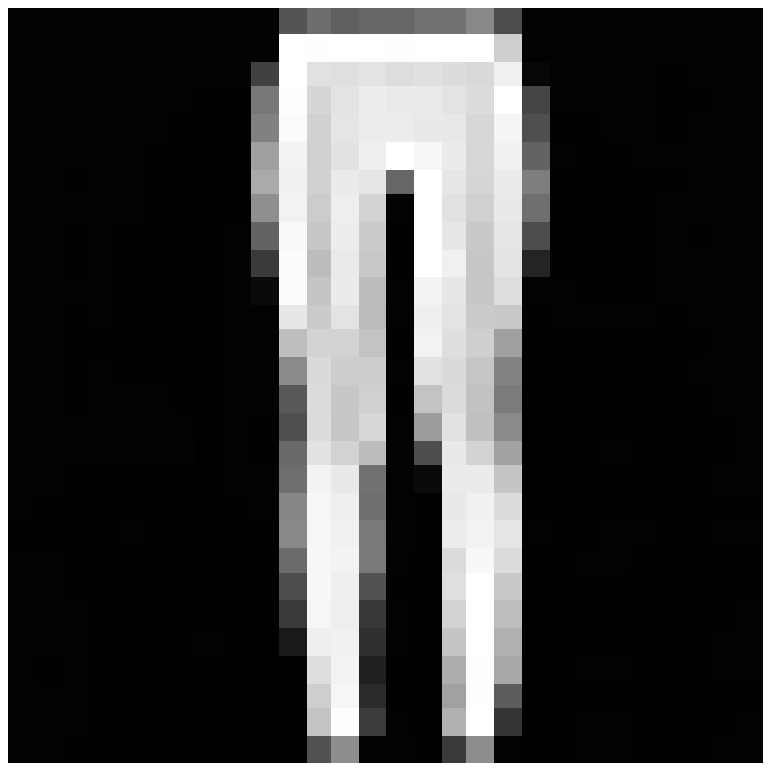}\includegraphics[width=0.16\columnwidth]{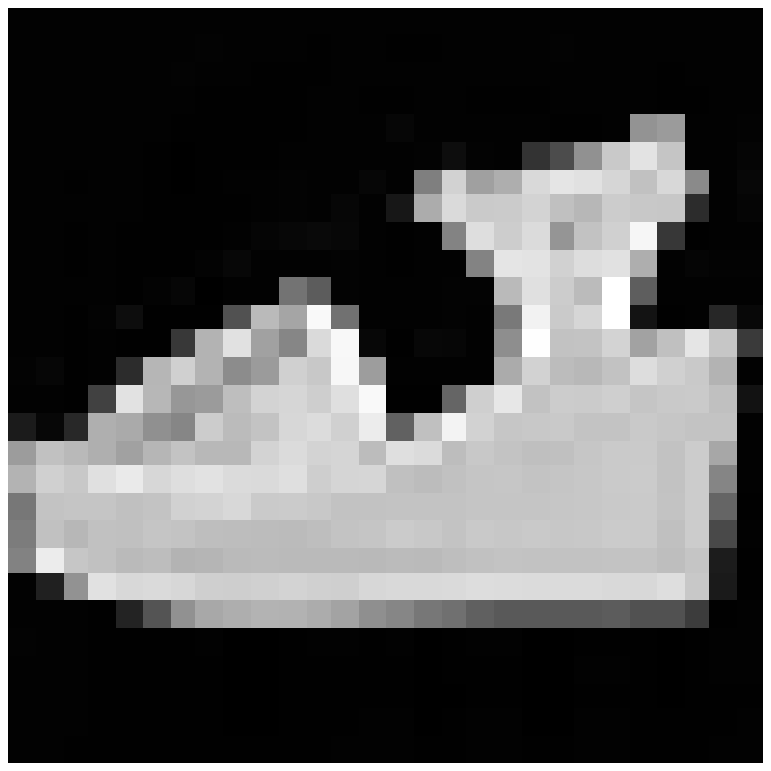}\includegraphics[width=0.16\columnwidth]{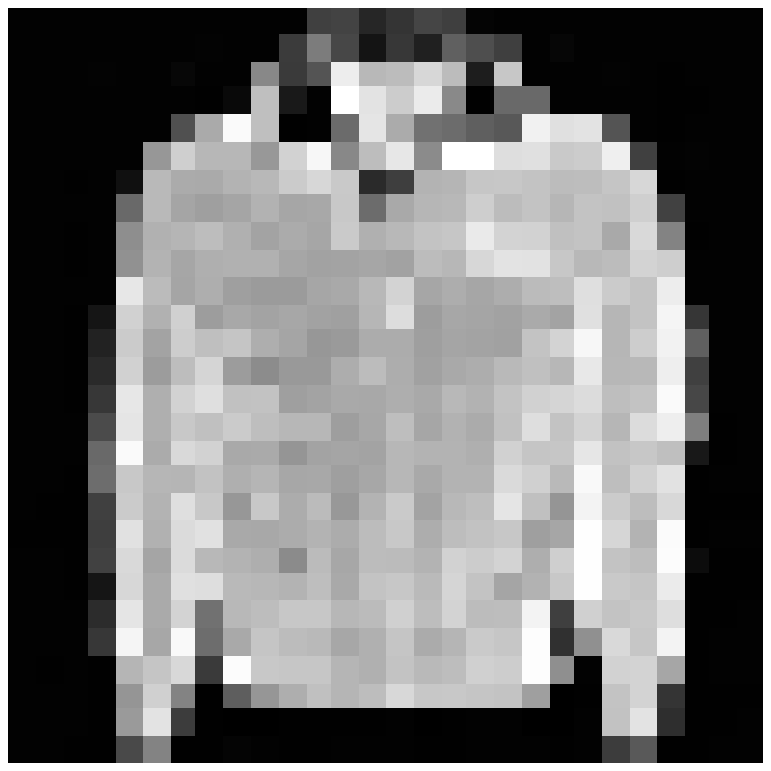}

\includegraphics[width=0.16\columnwidth]{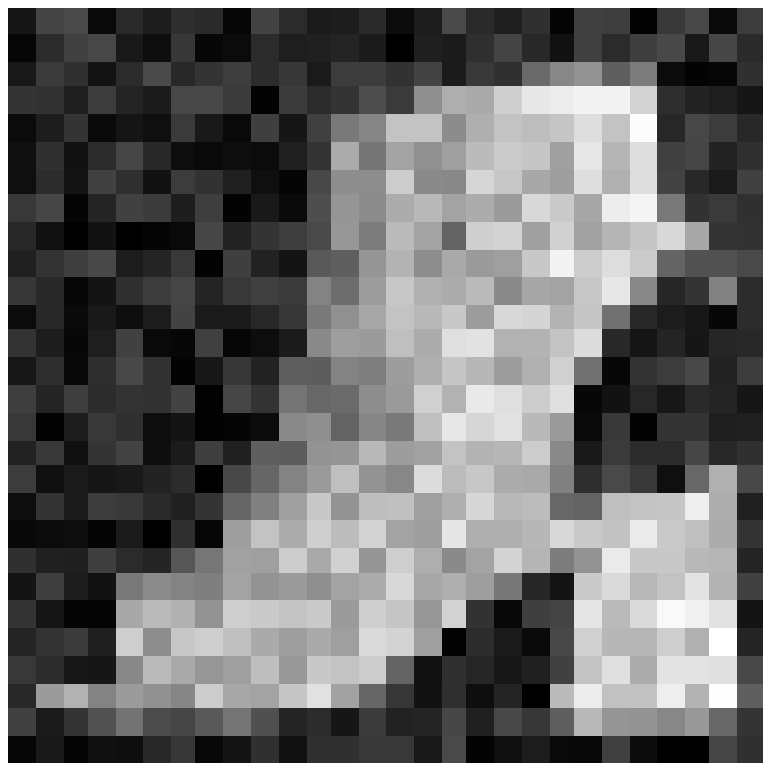}\includegraphics[width=0.16\columnwidth]{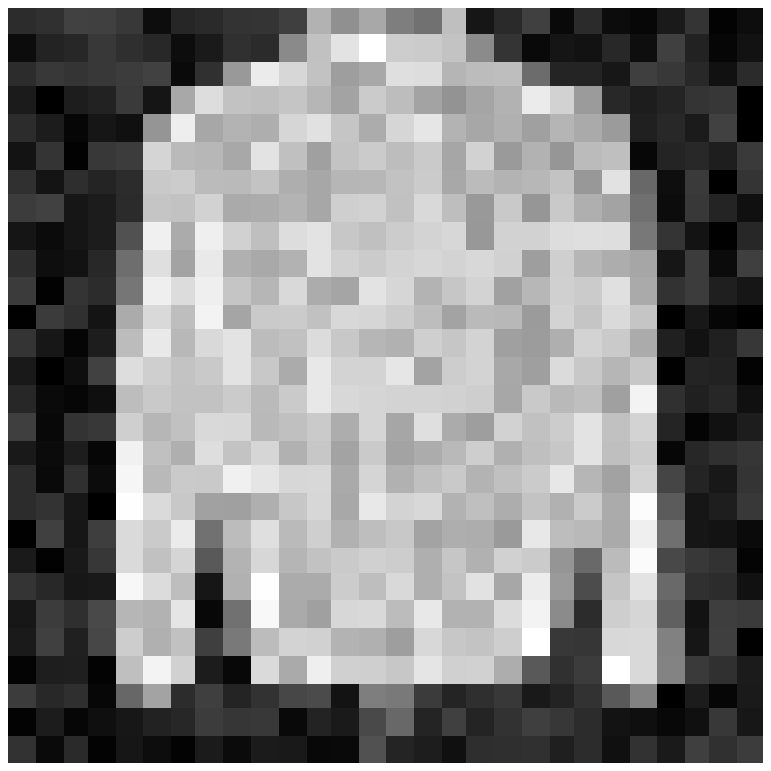}\includegraphics[width=0.16\columnwidth]{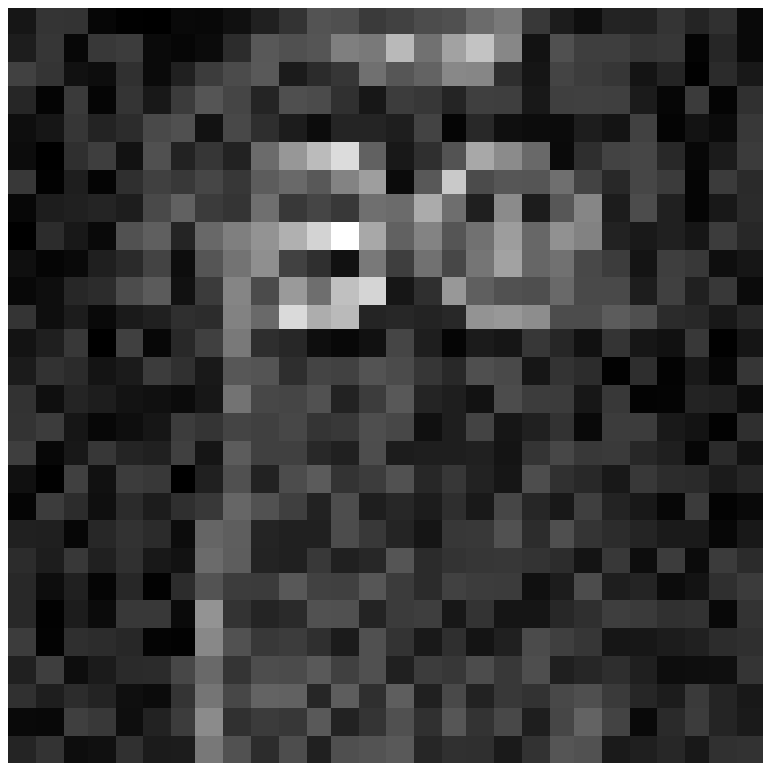}\includegraphics[width=0.16\columnwidth]{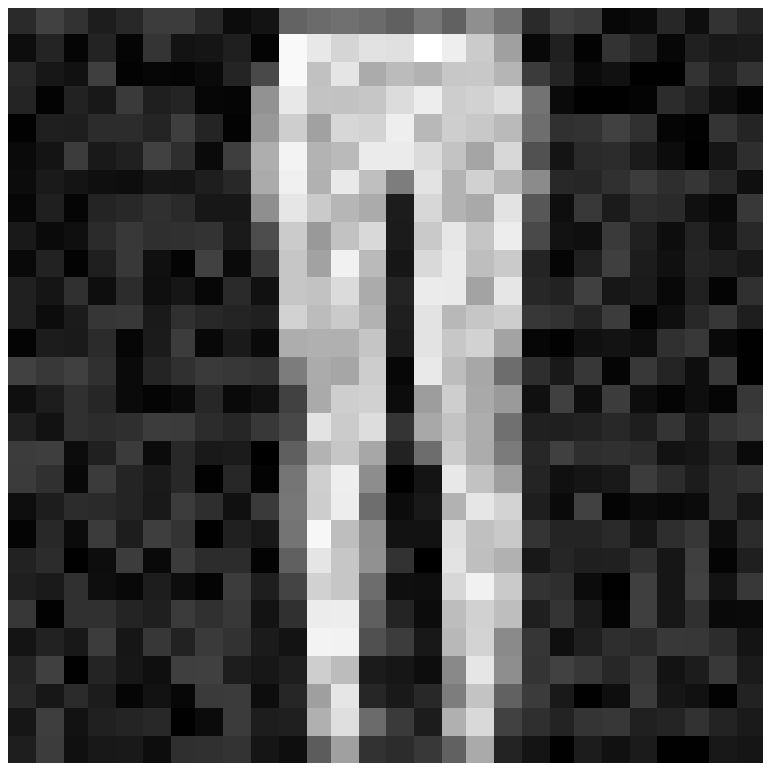}\includegraphics[width=0.16\columnwidth]{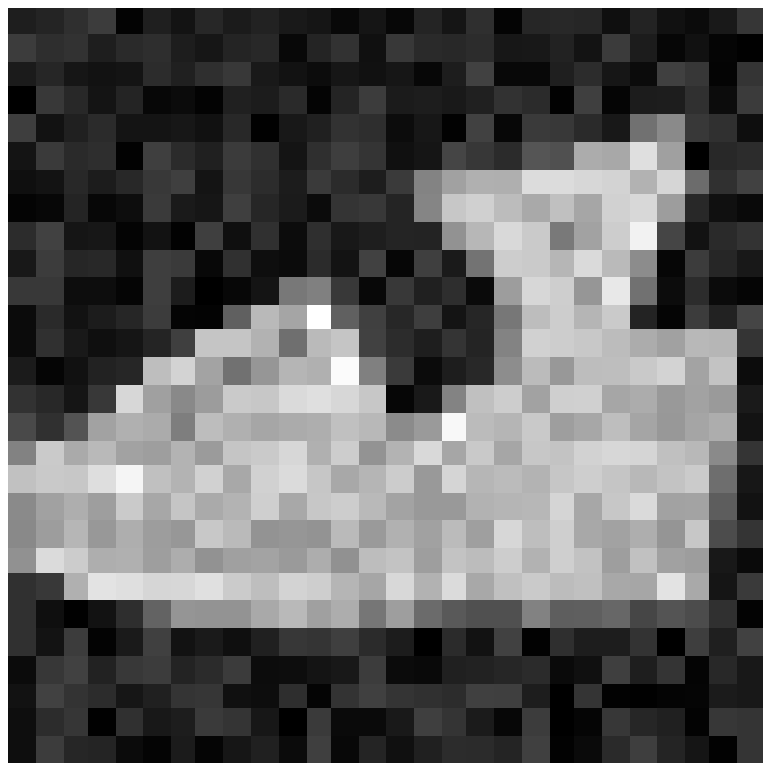}\includegraphics[width=0.16\columnwidth]{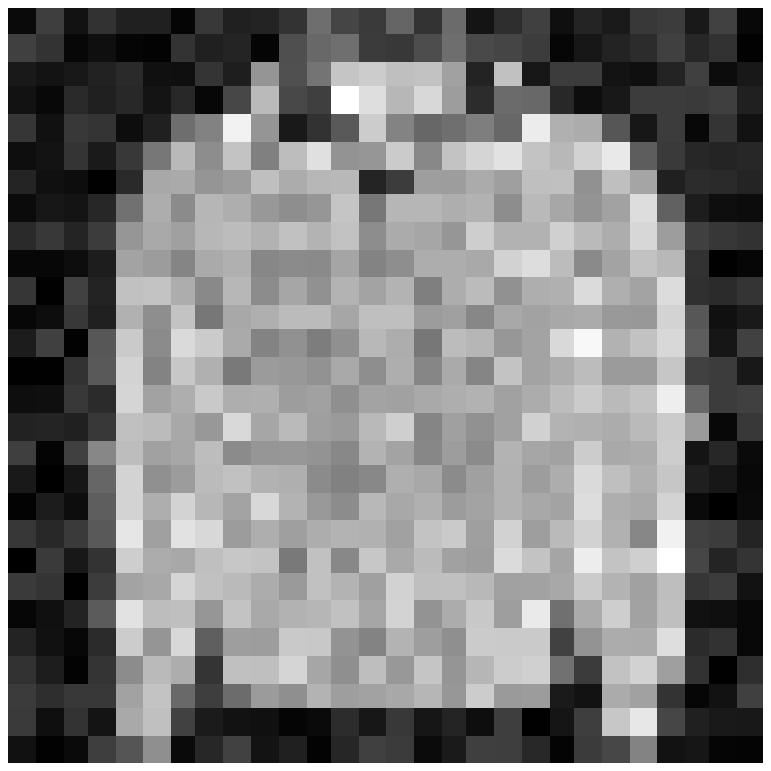}

\includegraphics[width=0.16\columnwidth]{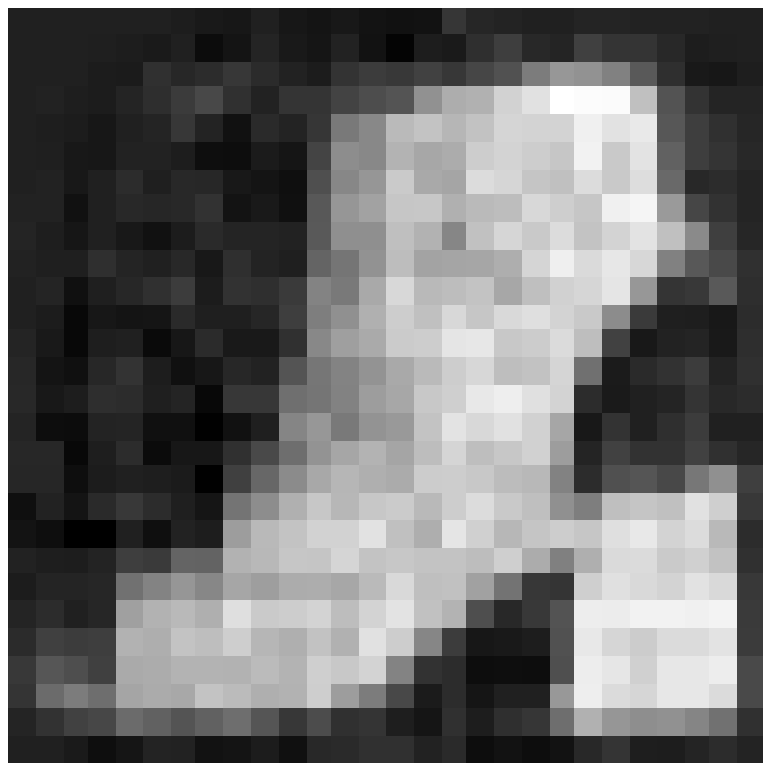}\includegraphics[width=0.16\columnwidth]{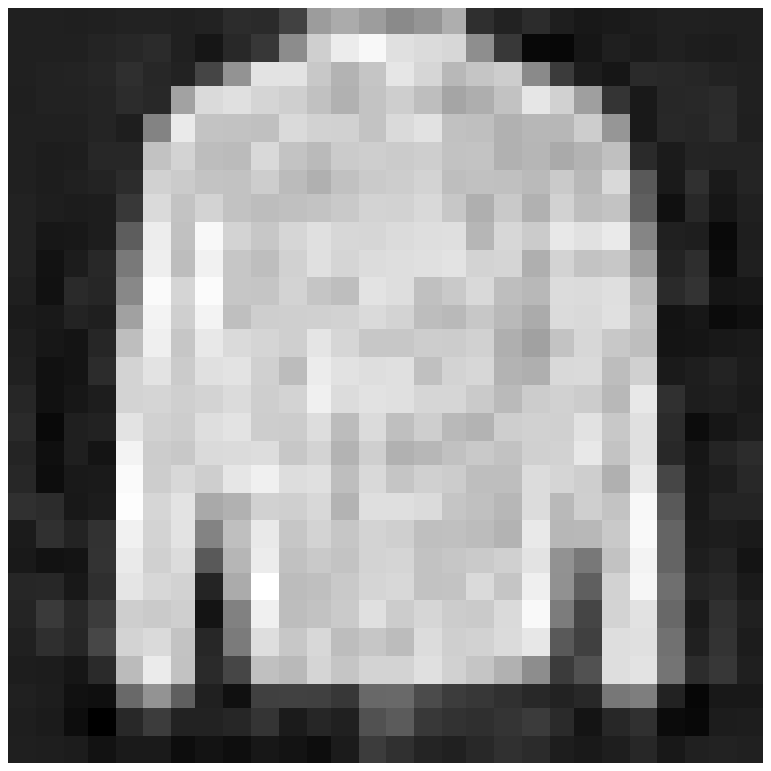}\includegraphics[width=0.16\columnwidth]{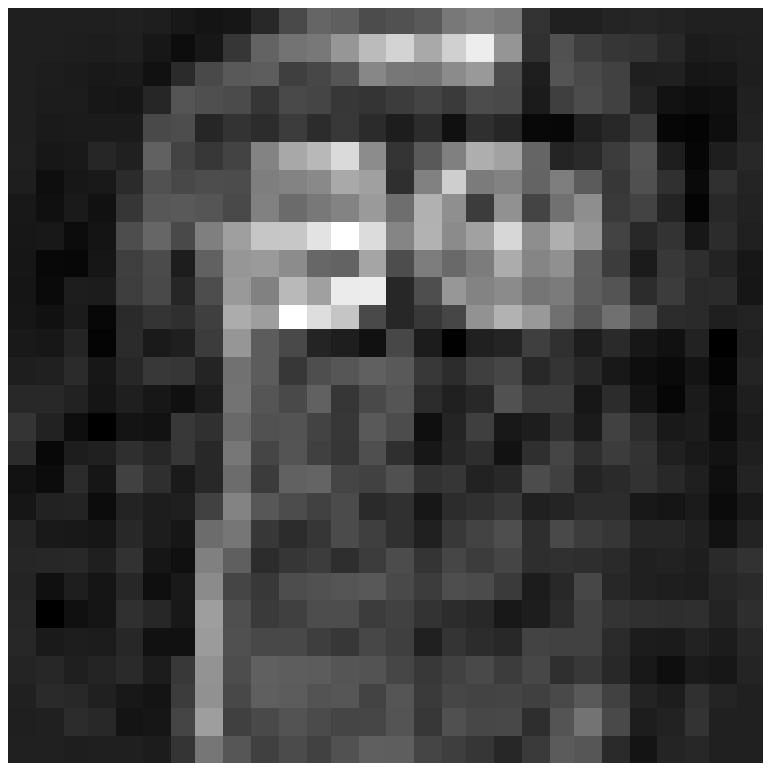}\includegraphics[width=0.16\columnwidth]{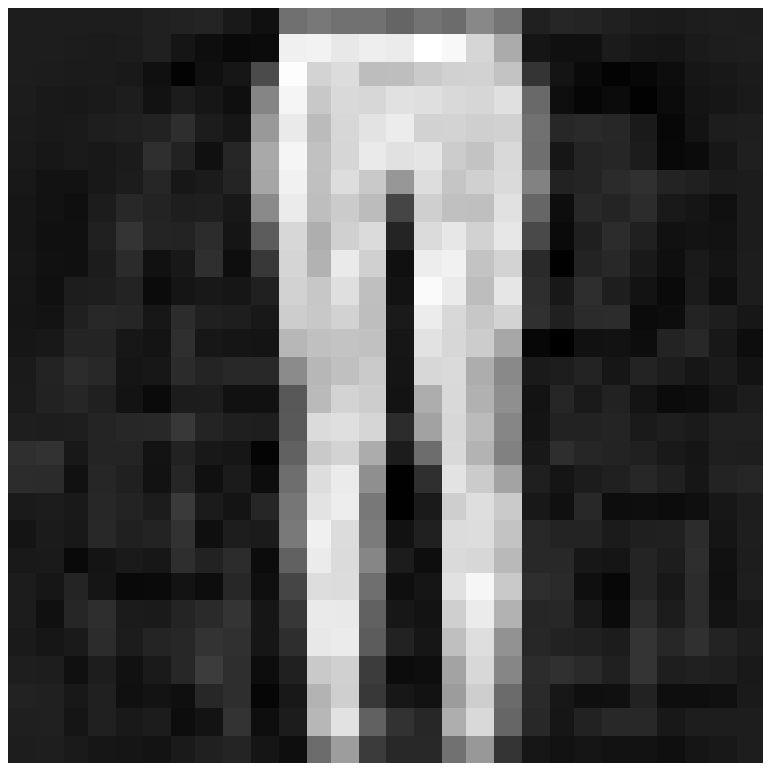}\includegraphics[width=0.16\columnwidth]{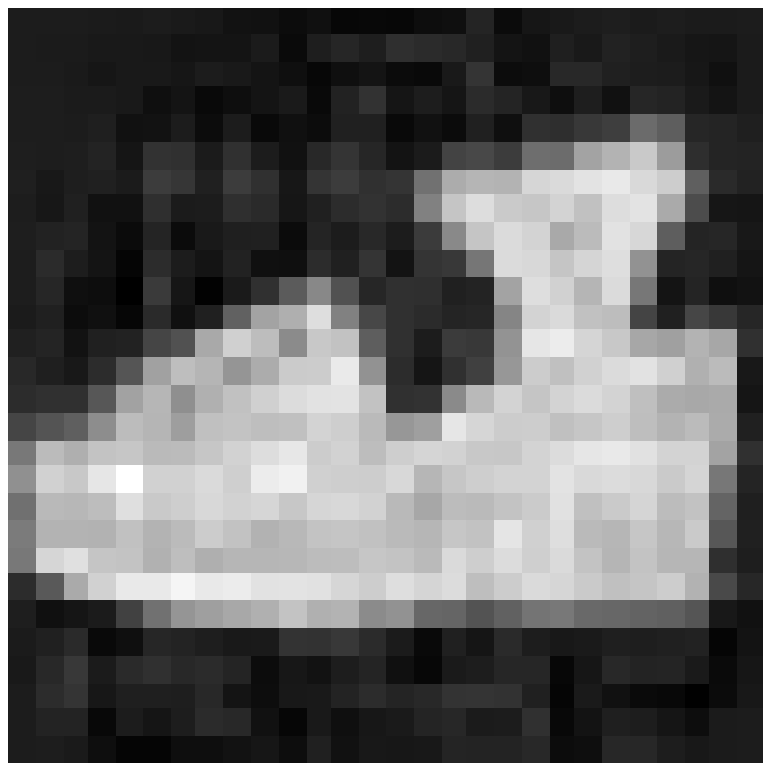}\includegraphics[width=0.16\columnwidth]{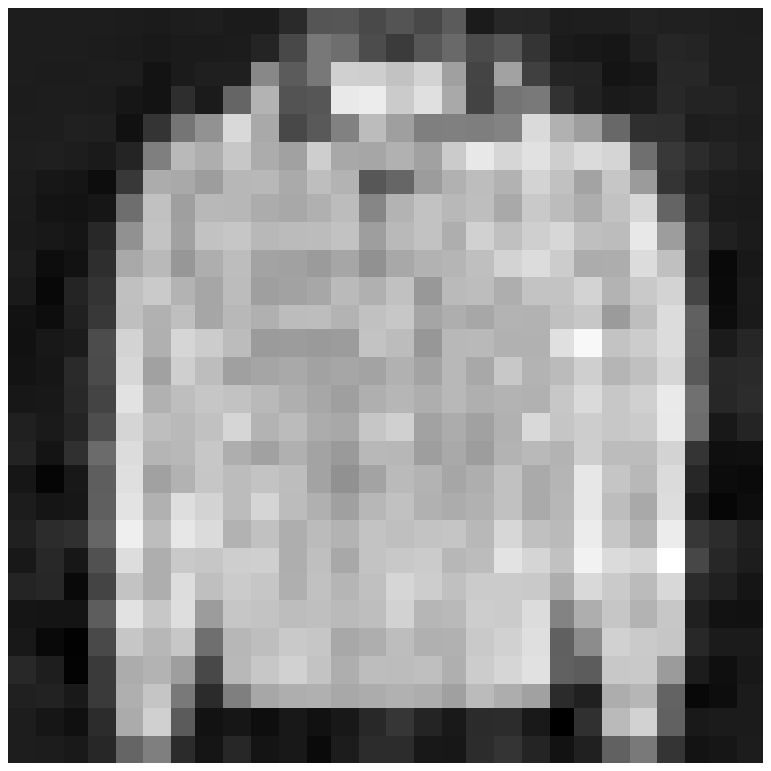}

\caption{ Ground truths (row 1), noisy data (row 2), and estimates (row 3) for 6 images from the test set. The least squares estimator was obtained based on $n = 13329$ samples.}\label{fig:samples_from_fashionMNIST} 
\end{figure}

\FloatBarrier

\section*{Acknowledgments}
The author gratefully acknowledges support by the International Research Training Group IGDK 1754 „Optimization and Numerical Analysis for Partial Differential Equations with Nonsmooth Structures“, funded by the German Research Council (DFG) and the Austrian Science Fund (FWF):[W 1244-N18]. Parts of the presented results originate from the author's dissertation. The author thanks his doctoral advisor, Karl Kunisch, for valuable discussions and for giving him the opportunity to work on this topic. Moreover, the author thanks Johannes Milz for valuable discussions and the idea for the proof of \Cref{lem:products_of_sub_gaussian_rv}.

\bibliography{main}
\bibliographystyle{plain}

\end{document}